\documentclass[a4paper,10pt,reqno]{amsart}

\usepackage[T1]{fontenc}			
\usepackage[utf8]{inputenc}
\usepackage[english]{babel}
\usepackage{amsmath, amsfonts, amssymb, amsthm, amscd, mathtools}
\usepackage{geometry}
\usepackage[square,numbers]{natbib}	
\usepackage{bbm, bm}
\usepackage{xspace} % to get the spacing after macros right
\usepackage{enumerate, enumitem}
\usepackage{hyperref}
\usepackage{xcolor}
\usepackage{dsfont}
\usepackage{graphicx}
\usepackage{subfig}
\usepackage{float}
\usepackage{caption}
\usepackage{mathrsfs}

\newcommand{\cA}{{\ensuremath{\mathcal A}} }

\newcommand{\cO}{{\ensuremath{\mathcal O}} }

\newcommand{\bfe}{{\ensuremath{\mathbf e}} }

\newcommand{\bfzero}{{\ensuremath{\mathbf 0}} }
\newcommand{\bfone}{{\ensuremath{\mathbf 1}} }

\newcommand{\dd}{{\ensuremath{\mathrm d}} }
\newcommand{\de}{{\ensuremath{\mathrm e}} }

\newcommand{\dC}{{\ensuremath{\mathrm C}} }
\newcommand{\dD}{{\ensuremath{\mathrm D}} }

\newcommand{\dL}{{\ensuremath{\mathrm L}} }

\newcommand{\rE}{{\ensuremath{\mathscr{E}}} }

\newcommand{\rG}{{\ensuremath{\mathscr{G}}} }

\newcommand{\rV}{{\ensuremath{\mathscr{V}}} }

\newcommand{\R}{\mathbb{R}}

\newcommand{\ind}{\ensuremath{\mathbf{1}}}

\newcommand{\diag}{\ensuremath{\mathrm{diag}}}
\DeclarePairedDelimiter{\abs}{\lvert}{\rvert}
\DeclarePairedDelimiter{\norm}{\lVert}{\rVert}
\DeclarePairedDelimiterX{\inprod}[2]{\langle}{\rangle}{#1, #2}
\newcommand{\sumtwo}[2]{\sum_{\substack{#1 \\ #2}}} % sum with 2 lines
\newcommand{\maxtwo}[2]{\max_{\substack{#1 \\ #2}}} % max with 2 lines

\renewcommand{\epsilon}{\varepsilon}
\newcommand{\changetheta}
{
\let\temp\theta
\let\theta\vartheta
\let\vartheta\temp
}
\newcommand{\changephi}
{
\let\temp\phi
\let\phi\varphi
\let\varphi\temp
}

\theoremstyle{plain}
\newtheorem{theorem}{Theorem}[section]
\newtheorem*{theorem*}{Theorem}
\newtheorem{lemma}[theorem]{Lemma}
\newtheorem*{lemma*}{Lemma}
\newtheorem{proposition}[theorem]{Proposition}
\newtheorem*{proposition*}{Proposition}
\newtheorem{corollary}[theorem]{Corollary}

\theoremstyle{definition}
\newtheorem{definition}[theorem]{Definition}
\newtheorem{assumption}[theorem]{Assumption}

\theoremstyle{remark}
\newtheorem{remark}[theorem]{Remark}
\newtheorem*{remark*}{Remark}

\definecolor{darkviolet}{rgb}{0.58, 0.0, 0.83}

\definecolor{gre}{rgb}{0.03,0.50,0.03}

\definecolor{darkgreen}{rgb}{0,0.4,0}

\graphicspath{{IMG/}}

\numberwithin{equation}{section}

\hypersetup{pdftitle={Economic growth on networks},pdfauthor={Alessandro Calvia, Fausto Gozzi, Marta Leocata, Georgios I. Papayiannis, Anastasios Xepapadeas, Athanasios N. Yannacopoulos},%
pdftex,pdfstartview={XYZ null null 1.4},colorlinks=true,linkcolor=blue,urlcolor=blue}

\newcommand{\mail}[1]{\href{mailto:#1}{\normalfont\texttt{#1}}}

\makeatletter
\def\@setthanks{\vspace{-\baselineskip}\def\thanks##1{\@par##1\@addpunct.}\thankses}
\makeatother

\title[Economic growth on networks]{An optimal control problem with state constraints in a  spatio-temporal economic growth model on networks}

\author[{\tiny A.~Calvia} ]{Alessandro Calvia\textsuperscript{\MakeLowercase{a}}}
\thanks{\noindent \textsuperscript{a} Dipartimento di Scienze Economiche e Aziendali, Università di Parma, Parma, Italy.}

\author[{\tiny F.~Gozzi} ]{Fausto Gozzi\textsuperscript{\MakeLowercase{b}}}
\thanks{\noindent \textsuperscript{b} Dipartimento di Economia e Finanza, LUISS \emph{Guido Carli}, Roma, Italy.}

\author[{\tiny M.~Leocata} ]{Marta Leocata\textsuperscript{\MakeLowercase{c}}}
\thanks{\noindent \textsuperscript{c} Scuola Normale Superiore, Pisa, Italy.}

\author[{\tiny G.I.~Papayiannis} ]{Georgios I. Papayiannis~\textsuperscript{\MakeLowercase{d,h}}}
\thanks{\noindent \textsuperscript{d} Section of Mathematics, Department of Naval Sciences, Hellenic Naval Academy, Piraeus}

\author[{\tiny A.~Xepapadeas} ]{Anastasios Xepapadeas \textsuperscript{\MakeLowercase{e,f}}}
\thanks{\noindent \textsuperscript{e} Department of International and European Economic Studies, Athens University of Economics and Business, Athens, Greece.
\\
\noindent \textsuperscript{f} Dipartimento di Scienze Economiche, Università di Bologna, Bologna, Italy.}

\author[{\tiny A.N.~Yannacopoulos} ]{Athanasios N. Yannacopoulos \textsuperscript{\MakeLowercase{g,h}}}
\thanks{\noindent \textsuperscript{g} Department of Statistics, Athens University of Economics and Business, Athens, Greece.
\\
\textsuperscript{h} Stochastic Modelling and Applications Laboratory, Athens University of Economics and Business.\\
\noindent \textit{E-mail addresses}: \mail{alessandro.calvia@unipr.it} (A. Calvia), \mail{fgozzi@luiss.it} (F. Gozzi), \mail{marta.leocata@sns.it} (M. Leocata), \mail{gpapagiannis@aueb.gr} (G.I.~Papayiannis), \mail{anastasio.xepapadeas@unibo.it} (A.~Xepapadeas), \mail{ayannaco@aueb.gr} (A.N.~Yannacopoulos).
\\
A.~Calvia, F.~Gozzi, and M.~Leocata are supported by the Italian Ministry of University and Research (MIUR), in the framework of PRIN project 2017FKHBA8 001 (The Time-Space Evolution of Economic Activities: Mathematical Models and Empirical Applications).
\\
A.~Calvia and M.~Leocata are members of the \emph{Gruppo Nazionale per l’Analisi Matematica, la Probabilità e le loro Applicazioni} (GNAMPA) of the \emph{Istituto Nazionale di Alta Matematica} (INdAM).
}
\begin{document}

\begin{abstract}
We introduce a spatial economic growth model where space is described as a network of interconnected geographic locations and we study a corresponding finite-dimensional optimal control problem on a graph with state constraints. Economic growth models on networks are motivated by the nature of spatial economic data, which naturally possess a graph-like structure: this fact makes these models well-suited for numerical implementation and calibration. 
The network setting is different from the one adopted in the related literature, where space is modeled as a subset of a Euclidean space, 
%(typically, the unit circle or the unit sphere), 
which gives rise to infinite dimensional optimal control problems. 
After introducing the model and the related control problem, we prove existence and uniqueness of an optimal control and a regularity result for the value function, which sets up the basis for a deeper study of the optimal strategies. Then, we focus on specific cases where it is possible to find, under suitable assumptions, an explicit solution of the control problem. Finally, we discuss the cases of networks of two and three geographic locations.
\end{abstract}

\maketitle

\noindent \emph{Key words}:
Optimal control problems;
Value function;
Graphs and networks;
Viscosity and regular solutions of HJB equations;
Bilateral viscosity solutions;
Spatial economic growth models;
AK production function.

\bigskip

\noindent \emph{AMS 2020:}
34H05, %(Control problems involving ordinary differential equations),
49K15, %(Optimality conditions for problems involving ordinary differential equations),
49L20, %(Dynamic programming in optimal control and differential games),
49L25, %(Viscosity solutions to Hamilton-Jacobi equations in optimal control and differential games),
91B62, %(Economic growth models),
93C15. %(Control/observation systems governed by ordinary differential equations).

\bigskip

\tableofcontents

\section{Introduction}
\label{sec:introduction}

\subsection{Motivation of the paper}

The classical theory of economic growth is related, from the technical viewpoint, to the study of optimal control problems with state and control constraints.
Such problems were studied in various papers in the mathematical and economic literature; we refer the reader to the book by~\citet{barro2004economic}, on the economic side, and to the books by~\citet{seierstad1987:optctrl} and by~\citet{yongzhou1999:stochctrl}, on the mathematical side.

In recent years, some research papers dealt with spatial heterogeneity of economic growth, which is due to the geographic distribution of the relevant economic variables. In the literature, various dynamic economic models of growth appeared, in which a spatial dimension is explicitly taken into account.
Most of these models considered a continuum of geographic locations distributed on a subset of a Euclidean space (typically, the unit circle or the unit sphere) and a state equation given by a parabolic Partial Differential Equation (PDE). From a mathematical point of view, this choice led to the study of infinite-dimensional optimal control problems with state constraints, see, e.g., \citet{boucekkine:spatialAK, boucekkine2009bridging, BFFGjoeg, brito2004dynamics, CFG2021, fabbri2016geographical, gozzileocata2022:growth, xepapadeas2016spatial, xepapadeas2023spatial}.

One of the main drawbacks of these models is that, except for special cases where the explicit form of the value function can be found, it is very difficult to get information on the regularity of the value function and, consequently, to prove verification theorems and to find optimal feedback strategies.
Moreover, difficulties arise in numerical simulations due to the curse of dimensionality and it is also nontrivial to model the spatial movement of the economic variables in such a setting.

In light of these facts, it appears reasonable and interesting to study models of economic growth where the space variable is described as a network of interconnected geographic locations. Mathematically speaking, by network we mean a weighted graph, whose nodes are the locations (such as cities, regions, etc.), arcs are the connections between some of them, and weights are given importance levels of each connection.
The weighted graph structure, in contrast to models with a continuum of locations, gives more realism to capital transport and matches the nature of spatial economic data, which have exactly such a network feature (see, e.g.,~\citet{allen2014trade}). This fact makes these models well-suited for numerical implementation and calibration. 

In this paper, we consider an economic growth problem on networks, where at each site (i.e., node) there is capital creation and consumption. We consider the case where the production function is linear, i.e., the so-called AK production function, a choice in line with most of the aforementioned literature. We model capital mobility through a linear operator $L$, which can be, for instance, the discrete Laplacian (cf. Remark~\ref{rem:lapl}). The choice of a general linear operator allows for more flexibility of our model, as opposed to the ones studied in the previously mentioned literature. In particular, it allows us to take into account different economic phenomena, such as capital concentration or dispersion (see, e.g., \citet{xepapadeas2016spatial}).

We study an optimal control problem where a centralized social planner aims at finding the optimal consumption path at each geographic location to maximize a given aggregate discounted utility. Clearly, capital levels must remain non-negative at each time and at each node of the graph, and hence we face an optimal control problem in $\R^n$ with state constraints, where $n$ is the number of nodes in the graph.

\subsection{Main contributions of the paper}
In this paper we establish some results (of which we give a brief summary below) for the class of economic growth models on networks that we introduce. These results are key towards a complete analysis of such models; our ultimate goals are to characterize the value function of the control problem and the optimal paths of the state and the control variables (i.e., capital and consumption), and to uncover how they depend on the network structure. These goals require a deeper analysis of the state constrained optimal control problem that we face, mainly because standard results, as the ones given in~\citet{soner:optcontrol1}, cannot be applied or directly generalized. Indeed, we consider linear state dynamics for capital movements across the network and an unbounded utility of consumption. This is in contrast with the cases usually considered in the literature, which assume the vector fields in the state equations and the running gain/cost functions to be bounded and Lipschitz. In particular, we consider a strictly increasing and concave utility function, thus neither linear nor quadratic, a fact that prevents us from applying a more direct approach to solve the optimization problem. 
Thus, to achieve our goals new techniques need to be developed, which is a task left for future research.

Nonetheless, as anticipated above, we prove in this paper the following important results: 
\begin{enumerate}[label=(\roman*)]
  \item\label{visc} We show in Theorem~\ref{teo:viscosity_bilateral} that the value function is a bilateral viscosity solution of a Hamilton-Jacobi-Bellman (HJB) equation (in the sense of~\citet[Definition~2.27]{bardi1997optimal}, i.e., it is a viscosity solution of the HJB and of its opposite) in the positive orthant of $\R^n$, where $n$ is the number of nodes in the graph. We also prove in~Theorem \ref{teo: regularity} that the value function is continuously differentiable in the interior of the positive orthant of $\R^n$.
  \item\label{feedback} We give the optimal feedback map in the interior of the positive orthant of $\R^n$, cf. Corollary~\ref{cor:feedbackopen}.
  \item\label{specific} In some specific cases, where we describe more precisely capital mobility on the network, we write explicitly the value function, under suitable conditions on the data, cf. Equation~\eqref{cond:V_Vb0}, by studying an associated auxiliary optimal control problem. See Proposition~\ref{PROP-AAA} and Theorem~\ref{th:Vaux}.
  \item In the specific cases mentioned above, we establish a result on the long run behavior of optimal capital and consumption paths, cf. Proposition~\ref{prop:Bproperties} and Theorem~\ref{th:stst}. We also analyze in Proposition~\ref{prop:lambda0dec} the behavior of the growth rate with respect to the connectivity of the graph.
  \item We study the two-nodes case with symmetric graph weights, where we further characterize the properties of the value function in the cases in which it can be explicitly computed, cf. Proposition~\ref{prop:V_Vb0_equiv} and Proposition~\ref{prop:max_V_2_node}. Finally, we discuss some numerical simulations in the three-nodes case with symmetric graph weights.
\end{enumerate}

Our main result is the regularity, i.e., continuous differentiability, of the value function in the interior of the positive orthant of $\R^n$.
On the one hand, this allows us to obtain a feedback characterization of optimal controls in the interior of the positive orthant of $\R^n$, which is crucial to qualitatively study and numerically approximate optimal paths.
On the other hand, this extends earlier results by~\citet{cannarsasoner1987:singularities} (see also~\citet{bardi1997optimal, cannarsasinestrari2004}, and~\citet{freni2008optimal} in a economic growth problem) to cases where the utility function may be unbounded at infinity and/or near zero.

To find the explicit form of the value function, as mentioned above in point~\ref{specific}, and to study the corresponding optimal paths, we adapt the ideas used in the infinite-dimensional case analyzed in~\citet{BFFGjoeg} and in~\citet{CFG2021}. 
%The present case is partly easier since it is finite dimensional, and partly more difficult since the matrix $L$ is far to be diagonal like in the above quoted cases. 
It is worth noting that we provide a sufficient condition in Proposition~\ref{PROP-AAA} to have explicit solutions of the HJB equation related to the optimal control problem in terms of a (possibly nonlinear) eigenvalue problem. This result and the condition given in~\eqref{cond:V_Vb0}, do not have equivalent counterparts in the infinite-dimensional setting cited above.

Finally, we note that the results given in Section~\ref{sec:twonodes} and the numerical simulations discussed in Section~\ref{sec:threenodes} provide some starting ideas to understand the connection between the graph structure and the qualitative properties of the solution of the optimization problem. We believe that it is interesting to develop this study, which is left for future research.

We conclude the introduction by outlining the plan of the paper. In Section~\ref{sec:opt_ctrl} we introduce the economic growth model and the optimal control problem. In Section~\ref{sec:ctrl_and_V} we establish existence and uniqueness of the optimal control and some useful properties of the value function, which are then exploited in Section~\ref{sec:DPP_HJB} to prove our main results, outlined in points~\ref{visc} and~\ref{feedback} above.
In Section~\ref{sec:AKgraph} we detail the connection of the control problem with the network structure and we provide the results on the explicit solutions mentioned in point~\ref{specific} above. Finally, in Sections~\ref{sec:twonodes} and~\ref{sec:threenodes} we discuss the two-nodes and the three-nodes cases, respectively.

\subsection{Notation}
In this section we collect the main notation used in this research article.

%Throughout the paper the set $\N$ denotes the set of non-negative integers $\N = \{0, 1, \dots \}$, whereas the set of positive integers is denote by $\Z^+ = \{1, 2, \dots \}$.

%For a fixed metric space $E$, we denote by $d_E$ its metric and by $\dB_b(E)$ the set of real-valued bounded measurable functions on $E$. The symbol $\cB(E)$ indicates the Borel $\sigma$-algebra on $E$ and we denote by $\cP(E)$ the set of probability measures on $E$.

%For a fixed metric space $E$, we denote by $\dB_b(E)$ the set of real-valued bounded measurable functions on $E$. The symbol $\cB(E)$ indicates the Borel $\sigma$-algebra on $E$ and we denote by $\cP(E)$ the set of probability measures on $E$.

The symbol $\R^n$ denotes the usual $n$-dimensional Euclidean space of $n$-tuples of real numbers; its canonical basis is denoted by $\{e_1, \dots, e_n\}$. The set of real numbers is denoted by $\R$. The sets
\begin{align*}
\R^n_+ &\coloneqq \{x = (x_1, \dots, x_n) \in \R^n \colon x_i \geq 0, \, \forall i=1, \dots, n\},
\\
\R^n_{++} &\coloneqq \{x = (x_1, \dots, x_n) \in \R^n \colon x_i > 0, \, \forall i=1, \dots, n\},
\end{align*}
indicate the positive cone (or positive orthant) and the strictly positive cone (or strictly positive orthant) of $\R^n$, respectively. Whenever convenient for the sake of clarity, we denote by $[x]_i$ the $i$-th component of a vector $x \in \R^n$.
For any $x, y \in \R^n$, $\norm{x}$ is the Euclidean norm of $x$ and $\inprod{x}{y}$ is the inner product of $x$ and $y$. If $A$ is any matrix, $A^T$ denotes its transpose and $\norm{A}$ indicates its Frobenius norm.

If $f \colon \R^n \to \R$ is differentiable at a point $x \in \R^n$, $\dD f(x)$ denotes the gradient of $f$ at $x$. The notations $\dC(E)$ and $\dC^1(E)$ indicate the sets of continuous and continuously differentiable functions on a set $E \subseteq \R^n$, respectively. The set of locally integrable functions on $[0,+\infty)$ with values in $E \subseteq \R^n$ is denoted by $\dL^1_{\mathrm{loc}}([0,+\infty);E)$.

%Finally, with the word \emph{measurable} we refer to \emph{Borel-measurable}, unless otherwise specified.

\section{The optimal control problem}\label{sec:opt_ctrl}

In this section we introduce the optimal control problem under study.
We are given a \emph{network}, describing a finite number of geographically distributed locations, $n \geq 2$, where capital is created and there is consumption. In the first part of the paper, we are not going to introduce an explicit modelization of said network, as we will only be interested in providing general results on the optimization problem. Starting from Section~\ref{sec:AKgraph}, we model this network as a weighted graph and we will derive more specific results.
At each site $i = 1, \dots, n$, capital is created according to an AK production function: the output of the economy at location $i$ is $(\bar A_i - \delta_i) k_i$, where $k_i \in \R$ is the current level of capital, $\bar A_i > 0$ is the technological level of site $i$, and $\delta_i \geq 0$ parameterizes a capital deterioration effect. In order to spare parameters, we are given the coefficients\footnote{Positivity of $A_i$ is assumed to make the economic problem meaningful.} $A_i \coloneqq (\bar A_i - \delta_i) > 0$, $i = 1, \dots, n$, i.e., we assume that each technological level incorporates a capital deterioration effect.
We are also given
\begin{itemize}
\item An operator $L \colon \R^n \to \R^n$, modeling capital mobility on the network, that is, the fact that capital may enter or exit any location as a result of interactions between them;
\item An operator $N \colon \R^n \to \R^n$, that describes the impact of consumption levels on the state of the system, i.e., capital.
\end{itemize}

In what follows, we denote the capital and consumption levels at site $i = 1, \dots, n$ and time $t \geq 0$ by $K_i(t)$ and $C_i(t)$, respectively. We also introduce the vectors $K(t) \coloneqq (K_1(t), \dots, K_n(t))$, $C(t) \coloneqq (C_1(t), \dots, C_n(t))$, and the diagonal matrix $A = \diag(A_1, \dots, A_n)$.

Given the data above, a vector $k = (k_1, \dots, k_n) \in \R^n_{+}$ of initial capital endowments, and a consumption path $C \in \dL^1_{\mathrm{loc}}([0,+\infty); \R^n_+)$, we assume that the dynamics of $K$ are described by the following ODE
\begin{equation}\label{eq:capitalODE}
\left\{
\begin{aligned}
&\dfrac{\dd}{\dd t} K(t) = (L+A)K(t) - NC(t), \quad t \geq 0,\\
&K(0) = k \in \R^n_{+}.
\end{aligned}
\right.
\end{equation}
Standard results ensure that ODE~\eqref{eq:capitalODE} has a unique solution, denoted by $K^{k,C}$, which verifies
\begin{equation}\label{eq:solcapitalODE}
K^{k,C}(t) = \de^{t(L+A)}k - \int_0^t \de^{(t-s)(L+A)} NC(s) \, \dd s, \quad t \geq 0.
\end{equation}
The following assumption will be in force throughout the paper.
\begin{assumption}\label{ass:main}
\mbox{}
\begin{enumerate}[label=(\roman*)]
\item\label{ass:L} $L$ is a linear operator generating a positive $\dC_0$-semigroup $(\de^{tL})_{t \geq 0}$, i.e., $\de^{tL}(\R^n_+) \subseteq \R^n_+$, for all $t \geq 0$.
\item\label{ass:Npos} $N$ is a positivity preserving linear operator, i.e., $N(\R^n_+) \subseteq \R^n_+$. Moreover, $N \bfe_i \neq \bfzero$, for all $i = 1, \dots, n$.
\end{enumerate}
\end{assumption}

Thanks to Assumption~\ref{ass:main}-\ref{ass:L}, the operator $L$ is represented by an $n \times n$ matrix, still denoted by $L$, with entries $\ell_{ij}$, $i,j = 1, \dots, n$. This assumption also ensures that $(\de^{tL})_{t \geq 0}$ is a positive linear system, that is, for every non-negative $k \in \R^n_+$, we have that $\de^{tL} k \in \R^n_+$, for all $t \geq 0$, cf.~\citep[Definition~2]{farinarinaldi2000}. This is equivalent to the requirement that $L$ is a Metzler matrix, i.e., $\ell_{ij} \geq 0$, for all $i \neq j$, $i,j = 1, \dots, n$, cf.~\citep[Theorem~2]{farinarinaldi2000}. Since $A$ is diagonal, also $L+A$ is a Metzler matrix, that is, $(\de^{t(L+A)})_{t \geq 0}$ is a positive linear system. This implies that, if we choose a zero-consumption path, i.e., if $C(t) = 0$, for all $t \geq 0$, then the solution~\eqref{eq:solcapitalODE} to~\eqref{eq:capitalODE} is non-negative for every non-negative initial capital endowment. Positive linear systems enjoy various important properties (see, e.g.,~\citep{farinarinaldi2000}), that are used throughout the paper and recalled whenever useful.

Assumption~\ref{ass:main}-\ref{ass:Npos} ensures that operator $N$ is represented by an $n \times n$ matrix, still denoted by $N$, with entries $n_{ij} \geq 0$, for all $i,j = 1, \dots, n$. Indeed, invariance of the closed positive orthant with respect to $N$ is equivalent to $N$ being a non-negative matrix. Moreover, the second part of this hypothesis entails that, for each $i = 1, \dots, n$, there exists $j = 1, \dots, n$ such that $n_{ij} > 0$. From an economic point of view, this means that each consumption level affects on the evolution of capital in one node at least.

We now proceed to introduce the optimal control problem that we aim to study. Given a discount factor $\rho > 0$ and, for fixed $\gamma > 0$, let $u_\gamma \colon \R_+ \to \R \cup \{-\infty\}$ be a utility function of the form
\begin{equation}\label{eq:utility}
u_\gamma(x) =
\begin{dcases}
\dfrac{x^{1-\gamma}}{1-\gamma}, &\text{if } \gamma \in (0,1) \cup (1,+\infty), \\
\log(x),						&\text{if } \gamma = 1,
\end{dcases}
\qquad x \geq 0.
\end{equation}
Consider a social planner, who aims at controlling consumption levels $C$ to maximize the gain functional
\begin{equation}\label{eq:Jfunct}
J( C) \coloneqq \int_0^{+\infty} \de^{-\rho t} \left(\sum_{i=1}^n p_i u_\gamma(C_i(t))\right) \dd t, \quad C \in \cA_+(k), \, k \in \R^n_+,
\end{equation}
where $p_i > 0$, $i = 1, \dots, n$, are given weights and the set of admissible controls $\cA_+(k)$ is defined as
\begin{equation}\label{eq:adctrl}
\cA_+(k) \coloneqq \{C \in \dL^1_{\mathrm{loc}}([0,+\infty); \R^n_+), \text{ s.t. } K^{k, C}(t) \in \R^n_+, \forall t \geq 0\}.
\end{equation}
We are, thus, interested in non-negative controls (i.e., consumption plans) that keep the capital levels non-negative.
The value function of this optimal control problem is
\begin{equation}\label{eq:ctrlpb_det}
V(k) = \sup_{C \in \cA_+(k)} J( C), \quad k \in \R^n_+.
\end{equation}
In the following, we denote by $U_\gamma \colon \R^n_+ \to \R \cup \{-\infty\}$, $\gamma > 0$, the function
\begin{equation}\label{eq:Ugamma}
U_\gamma(x) \coloneqq \sum_{i=1}^n p_i u_\gamma(x_i), \quad x \in \R^n_+,
\end{equation}
so that the gain functional introduced in~\eqref{eq:Jfunct} can be written as
\begin{equation*}
J(C) = \int_0^{+\infty} \de^{-\rho t} U_\gamma(C(t)) \dd t, \quad C \in \cA_+(k), \, k \in \R^n_+.
\end{equation*}

\begin{remark}\label{rem:lapl}
A natural choice for $L$ is the so-called \emph{discrete Laplacian}, that we adopt starting from Section~\ref{sec:AKgraph}. A possible choice for $N$, different from the identity matrix, is to consider $N = \mathrm{diag}(p_1, \dots, p_n)$, where $p_1, \dots, p_n$ are the weights appearing in functional $J$ introduced in~\eqref{eq:Jfunct}. These weigths may be related to the population level at each location $i = 1, \dots, n$ or other subjective criteria taken into account by the social planner.
\end{remark}

\begin{remark}\label{rem:utility}
Various results proved in Sections~\ref{sec:ctrl_and_V} and~\ref{sec:DPP_HJB} can be extended to the general case of utility functions that are increasing, concave, and either bounded from above or from below. To do so, one can adopt the approach adopted in~\citep{AB2017, bartaloni:phdthesis}.
\end{remark}

\section{Existence of an optimal control and properties of the value function}\label{sec:ctrl_and_V}
In this section we are going to prove that, under suitable assumptions, there exists a unique optimal control for the optimization problem introduced in Section~\ref{sec:opt_ctrl}. This is a fundamental result that will be used in Section~\ref{sec:DPP_HJB}, to characterize the value function $V$, given in~\eqref{eq:ctrlpb_det}, as a viscosity solution of a suitable HJB equation.
Then, we are going to establish important properties of the value function.

To begin with, we note that Assumption~\ref{ass:main}-\ref{ass:L} ensures that also $(L+A)^T$ is a Metzler matrix, and thus $(\de^{t(L+A)^T})_{t \geq 0}$ is a positive linear system, too. This implies (cf.~\citep[Theorem~11]{farinarinaldi2000}) that the dominant eigenvalue of $(L+A)^T$ is unique and real, called the \emph{Frobenius eigenvalue} and denoted by $\lambda_\star \in \R$. 

We introduce the following assumption, that will be used through Section \ref{sec:ctrl_and_V} and Section \ref{sec:DPP_HJB}.
\begin{assumption}\label{ass:rho_Frob}
\mbox{}
\begin{enumerate}[label=(\roman*)]
\item\label{ass:Frob} There exists an eigenvector $b^\star \in \R^n_{++}$ of $(L+A)^T$, with eigenvalue $\lambda_\star > 0$.
\item\label{ass:rho} $\rho>\lambda_\star(1-\gamma)$.
\end{enumerate}
\end{assumption}
In the Section \ref{sec:AKgraph}, certain assumptions will be determined, which are sufficient conditions for the assumption \ref{ass:rho_Frob}.
\begin{remark}
Positivity of $(\de^{t(L+A)^T})_{t \geq 0}$ grants the existence of positive eigenvectors (i.e., in $\R^n_+$) associated to the Frobenius eigenvalue, called \emph{Frobenius eigenvectors}. However, this fact alone is not enough to guarantee that Assumption~\ref{ass:rho_Frob}-\ref{ass:Frob} is satisfied, since $\lambda_\star$ may be nonpositive. It is also worth recalling that eigenvectors associated to eigenvalues different from the Frobenius eigenvalue cannot be strictly positive, cf.~\citep[Theorem~11]{farinarinaldi2000}.
\end{remark}
Our first result of this section concerns the existence and uniqueness of an optimal strategy with finite utility. We omit the proof, which can be carried out as in~\citep[Theorem~3.1]{freni2006existence}.

\begin{proposition}\label{prop: existence_optimal_control}
The value function $V$ is finite and for every initial state $k \in \R^n_+$, there exists a unique optimal strategy $\widehat C\in \mathcal{A}_+(k)$ for problem~\eqref{eq:ctrlpb_det}.
\end{proposition}

\begin{remark}
Proposition~\ref{prop: existence_optimal_control} can be extended to more general utility functions using the approach adopted in~\citep{AB2017, bartaloni:phdthesis}.
\end{remark}

In the next proposition, we show important regularity properties of value function $V$, introduced in~\eqref{eq:ctrlpb_det}.
We recall that, given $E \subseteq \R^n$, $v\in \dC(E)$, and $k \in E$, the \emph{superdifferential} of $v$ at $k$ is the set
\begin{equation*}
\dD^+v(k) \coloneqq \left\{q \in \R^n \colon \limsup_{x \to k, \, x \in E} \dfrac{v(x) - v(k) - \inprod{p}{x-k}}{\norm{x-k}} \leq 0\right\}.
\end{equation*}

\begin{proposition}\label{prop_concavity_V}
\mbox{}
\begin{enumerate}[label=(\roman*)]
\item\label{prop:V_incr} $V$ is component-wise increasing in $\R^n_+$, i.e., for all $i = 1, \dots, n$ and for all $k, k' \in \R^n_+$, such that $k_i \leq k'_i$ and $k_j = k'_j$, $j \neq i$, it holds $V(k) \leq V(k')$.
\item\label{prop:V_conc} $V$ is concave in $\R^n_+$ and locally Lipschitz continuous on $\R^n_{++}$.
\item\label{prop:V_superdiff} For any $k \in \R^n_{++}$, $\dD^+V(k)$ is non-empty and $\dD^+V(k) \subseteq \R^n_{++}$.
\item\label{prop:V_hom} If $\gamma \neq 1$, then $V$ is \mbox{$(1-\gamma)$-homogeneous}, i.e., $V(\alpha k)=\alpha^{1-\gamma}V(k)$, for each $\alpha>0$ and $k\in \R^n_+$.
\item\label{prop:V_cont} If $0 < \gamma < 1$, then the value function $V$ is continuous on $\R^n_+$.
\end{enumerate}
\end{proposition}

\begin{proof}
Points~\ref{prop:V_incr} and~\ref{prop:V_conc} can be shown as in~\citep[Proposition~4.2-(i), (iii)]{freni2008optimal}, exploiting the linearity of the state equation~\eqref{eq:capitalODE} and the concavity of $U_\gamma$, defined in~\eqref{eq:Ugamma}. Similarly, the proof of point~\ref{prop:V_hom} can be carried out as in~\citep[Proposition~4.2-(ii)]{freni2008optimal}.

We prove, next, point~\ref{prop:V_superdiff}. Fix $k \in \R^n_{++}$. By~\citep[Theorem~A.1.13]{cannarsasinestrari2004}, concavity of $V$ entails that $\dD^+V(k) \neq \emptyset$. It also implies that $\dD^+V(k)$ coincides with the definition of the semidifferential of $V$ at $k$ from convex analysis (cf.~\citep[Remark~3.3.2]{cannarsasinestrari2004}), i.e., $q \in \dD^+V(k)$ if and only if
\begin{equation}\label{eq:semidiff}
V(x) - V(k) \leq \inprod{q}{x-k}, \qquad \forall x \in \R^n_+.
\end{equation}
Suppose, by contradiction, that there exists $q \in \dD^+V(k)$ such that $q_i \leq 0$, for some $i = 1, \dots, n$. Since $V$ is component-wise increasing and concave, this implies (by~\eqref{eq:semidiff}) that the restriction of $V$ to the half-line $k + \alpha e_i$, $\alpha \geq 0$, is constant. We want to prove that this leads to a contradiction.

Let $C^\star \in \cA_+(k)$ be the optimal control for $k$. Fix $\alpha, \epsilon > 0$, define $\bar k = k + \alpha e_i \in \R^n_{++}$, and let $\bar C(t) = C^\star(t) + \ind_{[0,\epsilon]}(t) \delta e_i$, $t \geq 0$, where $\delta > 0$ is a constant that we need to determine later to ensure that $\bar C \in \cA_+(\bar k)$.

It is clear that $\bar C \in \dL^1_{\mathrm{loc}}([0,+\infty); \R^n_+)$. This control is admissible for $\bar k$ if and only if $K^{\bar k, \bar C}(t) \in \R^n_+$, for all $t \geq 0$. Using~\eqref{eq:solcapitalODE}, we get
\begin{equation*}
K^{\bar k, \bar C}(t) = K^{k, C^\star}(t) + \underbrace{\alpha \de^{t(L+A)}e_i - \delta \int_0^t \ind_{[0,\epsilon]}(s) \, \de^{(t-s)(L+A)} N e_i \, \dd s}_{\coloneqq R(t)}.
\end{equation*}
To prove admissibility, it suffices to show that $R(t) \in \R^n_+$, for all $t \geq 0$. Note that, if $t > \epsilon$, we have
\begin{equation*}
K^{\bar k, \bar C}(t) = K^{k, C^\star}(t) + \de^{(t-\epsilon)(L+A)}R(\epsilon).
\end{equation*}
Therefore, thanks to the positivity of semigroup $(\de^{t(L+A)})_{t \geq 0}$, it is enough to prove that $R(t) \in \R^n_+$, for all $0 \leq t \leq \epsilon$. Fix $t \in [0,\epsilon]$. We have that, for all $j = 1, \dots, n$,
\begin{align*}
&\phantom{\Longleftrightarrow}\inprod{K^{\bar k, \bar C}(t)}{e_j} \geq \alpha \inprod{\de^{t(L+A)}e_i}{e_j} - \delta \int_0^\epsilon \inprod{\de^{(t-s)(L+A)} N e_i}{e_j} \, \dd s \geq 0
\\
&\Longleftrightarrow \int_0^\epsilon \left \langle \de^{(t-s)(L+A)} \left[\frac{\alpha}{\delta\epsilon} \de^{s(L+A)} - N\right] e_i, e_j \right\rangle \, \dd s \geq 0
\end{align*}
Using again the fact that $(\de^{t(L+A)})_{t \geq 0}$ is a positive semigroup, the inequality above is true if $\left[\frac{\alpha}{\delta\epsilon} \de^{s(L+A)} - N\right] e_i \in \R^n_+$, for all $s \in [0,\epsilon]$. We have that, for all $j = 1, \dots, n$,
\begin{align*}
&\phantom{\Longleftrightarrow}\left \langle \left[\frac{\alpha}{\delta\epsilon} \de^{s(L+A)} - N\right] e_i, e_j \right\rangle \geq 0
\\
&\Longleftrightarrow [N e_i]_j \leq \frac{\alpha}{\delta\epsilon} \inprod{\de^{s(L+A)}e_i}{e_j} \leq \frac{\alpha}{\delta\epsilon} \norm{\de^{s(L+A)}} \leq \frac{\alpha}{\delta\epsilon} \de^{s\norm{L+A}} \leq \frac{\alpha}{\delta\epsilon} \de^{\epsilon\norm{L+A}}.
\end{align*}
Thanks to Assumption~\ref{ass:main}-\ref{ass:Npos}, this inequality is valid choosing
\begin{equation*}
\delta = \dfrac{\alpha \de^{\epsilon\norm{L+A}}}{\epsilon \max\limits_{j = 1, \dots, n} [Ne_i]_j} \leq \dfrac{\alpha \de^{\epsilon\norm{L+A}}}{\epsilon [Ne_i]_j}, \quad \forall j = 1, \dots, n.
\end{equation*}
Thus, $\bar C \in \cA_+(\bar k)$. Moreover, using the fact that the utility function $u_\gamma$, defined in~\eqref{eq:utility}, is strictly increasing, we have that $J(\bar C) > J(C^\star) = V(k)$. Therefore, $V(\bar k) > V(k)$, which contradicts the fact that $V$ should be constant on the half-line $k + \alpha e_i$, $\alpha \geq 0$.

Finally, we prove point~\ref{prop:V_cont}.
Continuity of $V$ on $\R^n_{++}$ is a consequence of local Lipschitz continuity on the same set. We have, thus, to prove continuity on the boundary.

Consider $k\in \partial \R^n_{+}$, and fix a unit vector $v\in \R^n$ such that $k+\mu v\in \R^n_{+}$, for $\mu$ small enough. As a consequence of~\citep[Theorem~10.3]{rockafellar} (see also the discussion below this reference), it is sufficient to show that
\begin{equation}\label{eq:continuity_sufficient_condition}
V(k+\mu_m v)\to V(k), \,\, \text{as}\,\, m\to \infty,
\end{equation}
for any sequence $(\mu_m)_{m \geq 1}$ converging to zero.
From Proposition~\ref{prop: existence_optimal_control}, we know that there exists a family of optimal controls $C_m\in \mathcal{A}_+(k+\mu_m v)$ for $k + \mu_m v$, i.e., $J(C_m)=V(k+\mu_m v)$, for all $m \geq 1$. From Lemma~\ref{lem:bound_c}, there exists $\Gamma > 0$, independent of $(\mu_m)_{m \geq 1}$ and $(C_m)_{m \geq 1}$, such that
\begin{equation}\label{eq:_bound_cm}
\int_0^{+\infty} \de^{-s\lambda_\star} \norm{C_m(s)} \, \dd s \leq \Gamma \inprod{k + \mu_m v}{b^\star} \leq \Gamma \norm{b^\star} \norm{k},
\end{equation}
where the last inequality holds for $m$ large enough.
Then, by Lemma~\ref{lem:bound_c_cons} there exists\footnote{The definition of the weighted spaces $\dL^{p}_{\lambda_\star}([0,+\infty); \R^n)$, $p \geq 1$, is given in~\eqref{eq:Lp}.} $C_0\in \dL^{1}_{\lambda_\star}([0,+\infty); \R^n)$, with $C_0(t) \in \R^n_+$, for all $t \geq 0$, such that, up to subsequences,
$ C_m^{1-\gamma}\rightarrow C_0^{1-\gamma}$, weakly in $\dL^{\frac{1}{1-\gamma}}_{\lambda_\star}([0,+\infty); \R^n)$.

We need to show that $C_0$ is admissible for $k$, i.e., $C_0\in \mathcal{A}_+(k)$. By~\eqref{eq:_bound_c_n_cons2} and by linearity of the semigroup, we get that, for all $i = 1, \dots, n$,
\begin{equation*}
\inprod{K^{k,C_0}(t)}{e_i} \geq \limsup_{m\to \infty} \inprod{K^{k+\mu_m v,C_m}(t)}{e_i} \geq 0,
\end{equation*}
whence we deduce that $C_0\in \mathcal{A}_+(k)$.
Finally, from~\eqref{eq:_bound_c_n_cons3} we get
\begin{equation*}
\lim_{m\to \infty} V(k+\mu_m v)=\lim_{m\to \infty} J(C_m)=J(C_0)\leq V(k).
\end{equation*}
Since $\mu \mapsto V(k+\mu v)$ is concave, the inequality implies~\eqref{eq:continuity_sufficient_condition}.
\end{proof}

\section{Dynamic Programming Principle and the HJB equation}\label{sec:DPP_HJB}
In this section we are going to show that the value function $V$, given in~\eqref{eq:ctrlpb_det}, satisfies both the Dynamic Programming Principle and the Backward Dynamic Programming Principle. A key result towards proving that $V$ is a bilateral viscosity solution (precise definitions are given below) of a suitable HJB equation.

The following lemma asserts that $V$ verifies the Dynamic Programming Principle. We omit the proof, since it is standard (see, e.g.,~\citep{bardi1997optimal}).

\begin{lemma}[Dynamic Programming Principle]\label{lem:DPP}
For all $k\in \R^n_+$ and $t>0$,
\begin{equation}\label{DPP}
V(k)=\sup_{C\in \mathcal{A}_+(k)}\left[\int_0^t \de^{-\rho s} U_\gamma(C(s)) \, \dd s + \de^{-\rho t}V(K^{k,C}(t))\right].
\end{equation}
If Assumption~\ref{ass:rho_Frob} is verified, then the supremum in~\eqref{DPP} is attained at $\widehat C$, where $\widehat C$ is optimal for $k \in \R^n_+$, i.e.,
\begin{equation}\label{DPP2}
V(k)=\int_0^t \de^{-\rho s} U_\gamma(\widehat C(s)) \, \dd s + \de^{-\rho t}V(K^{k,\widehat C}(t)).
\end{equation}
\end{lemma}

A similar result can be obtained for backward trajectories of the linear system~\eqref{eq:capitalODE}. We need, first, some definitions. For each $k\in \R^n_+$, we denote by $\mathcal A_-(k)$ the set of locally integrable strategies such that the positivity constraint is satisfied by the backward trajectory, namely,
\begin{equation*}
\mathcal A_-(k) \coloneqq \{C \in \dL^1_{\mathrm{loc}}((-\infty,0]; \R^n_+) \text{ s.t. } K^{k,C}(t) \in \R^n_+,\, \forall t\leq 0\},
\end{equation*}
where $K^{k,C}$ is the unique solution to the backward ODE,
\begin{equation*}
\begin{cases}
\dfrac{\dd}{\dd t}K(t)=(L+A)K(t)-NC(t), \quad t<0,\\
K(0)=k,
\end{cases}
\end{equation*}
which is given by~\eqref{eq:solcapitalODE}.

\begin{definition}
We say that $k_0$ is an optimal point if there exists an initial condition $k\in \R^n_+$, $t>0$, and an optimal strategy $\widehat C\in \mathcal A_+(k)$ such that $k_0=K^{k,\widehat C}(t)$. We denote by $\mathcal O$, the set of optimal points.
\end{definition}

\begin{lemma}[Backward Dynamic Programming Principle]\label{lem:dpp_backward}
For every $k\in\R^n_+$, $t > 0$, and $C\in \mathcal A_-(k)$ the value function satisfies the following inequality,
\begin{equation}\label{dpp_backward_1}
V(k)\leq \de^{\rho t}V(K^{k,C}(-t))-\int_0^t \de^{\rho s} U_\gamma(C(-s)) \, \dd s.
\end{equation}
Moreover, if $k_0\in \mathcal O$, there exists $\overline C\in \mathcal{A}_-(k_0)$ such that
\begin{equation}\label{dpp_backward_2}
V(k_0)= \de^{\rho t }V(K^{k_0,\overline C}(-t))-\int_0^t \de^{\rho s} U_\gamma(\overline C(-s)) \, \dd s.
\end{equation}
\end{lemma}

\begin{proof}
Consider $k\in\R^n_+$, $t > 0$, and $C\in \mathcal A_-(k)$. Set $k_0 \coloneqq K^{k,C}(-t)$, define $\widetilde C(s) \coloneqq C(s-t)$, with $0\leq s\leq t$, and extend it to zero for $s > t$. Then, using~\eqref{eq:solcapitalODE}, it is quite straightforward to prove that $K^{k_0,\widetilde C}(\tau) \in \R^n_+$, for $\tau \leq t$; moreover, by the positivity of $(\de^{t(L+A)})_{t \geq 0}$, $K^{k_0,\widetilde C}(\tau) \in \R^n_+$, for $\tau > t$. Therefore, $\widetilde C\in \mathcal A_+(k_0)$.
By applying the Dynamic Programming Principle, i.e.,~\eqref{DPP}, we get
\begin{equation*}
V(k_0)\geq \int_0^t \de^{-\rho s} U_\gamma(\widetilde C(s)) \, \dd s + \de^{-\rho t}V(K^{k_0,\widetilde C}(t)).
\end{equation*}
Multiplying both sides by $\de^{\rho t}$, recalling that $K^{k_0,\widetilde C}(t)=k$ and $k_0=K^{k,C}(-t)$, and making a change of variables in the integral, we conclude that
\begin{equation*}
\de^{\rho t} V(K^{k,C}(-t))\geq \int_0^t \de^{\rho s} U_\gamma(C(-s)) \, \dd s+V(k),
\end{equation*}
whence~\eqref{dpp_backward_1}. To prove~\eqref{dpp_backward_2}, let us fix $k\in \R^n_+$ and recall that, by Proposition~\ref{prop: existence_optimal_control},
there exists an optimal control $\widehat C \in \mathcal A_+(k)$. Set $k_0 \coloneqq K^{k,\widehat C}(t) \in \cO$, define $\overline C(s)=\widehat C(s+t)$, with $-t\leq s\leq 0$, and extend it to zero for $s < -t$. Reasoning as before, we have that $\overline C\in \mathcal{A}_-(k_0)$. Applying~\eqref{DPP2}, we get
\begin{equation*}
V(k)=\int_0^t \de^{-\rho s} U_\gamma(\widehat C(s)) \, \dd s + \de^{-\rho t} V(K^{k,\widehat C}(t)).
\end{equation*}
Therefore, making similar computations as above, we get~\eqref{dpp_backward_2}.
\end{proof}

The HJB equation associated to the optimal control problem~\eqref{eq:ctrlpb_det} is
\begin{equation}\label{eq:HJB}
\rho v(k) = \inprod{(L+A)k}{\dD v(k)} + H_\gamma(\dD v(k)), \quad k \in \R^n_+,
\end{equation}
where the Hamiltonian $H_\gamma \colon \R^n \to \R$ is defined, for all $q \in \R^n$, as
\begin{equation}\label{eq:hamilton}
H_\gamma(q) \coloneqq \sup_{c \in \R^n_+} \{-\inprod{Nc}{q} + U_\gamma(c)\}.
\end{equation}
Standard computations show that the Hamiltonian can be explicitly written as
\begin{equation}\label{eq:hamilton_expl}
H_\gamma(q) =
\begin{dcases}
\frac{\gamma}{1-\gamma} \sum_{i=1}^n p_i^{\frac 1\gamma} [N^T q]_i^{1-\frac 1\gamma}, &\text{if } q \in \R^n_{++}, \, \gamma \neq 1, \\
\sum_{i=1}^n p_i \left[\log\left(\dfrac{p_i}{[N^T q]_i}\right) - 1 \right], &\text{if } q \in \R^n_{++}, \, \gamma = 1, \\
+\infty, &\text{otherwise},
\end{dcases}
\end{equation}
with maximizer $\widehat c \colon \R^n_{++} \to \R^n_{++}$ given by
\begin{equation}\label{eq:Hamilt_maxim_appl}
\widehat c_i(q) = \left(\dfrac{p_i}{[N^T q]_i}\right)^{\frac 1\gamma}, \quad i = 1, \dots, n, \quad q \in \R^n_{++}.
\end{equation}
Note that $\widehat c$ is continuous and its strict positivity derives from the strict positivity of parameters $p_1, \dots, p_n$ and from Assumption~\ref{ass:main}-\ref{ass:Npos}.

Thanks to the dynamic programming principles stated in Lemma \ref{lem:DPP}, we are able to characterize the value function $V$ as a bilateral viscosity solution to~\eqref{eq:HJB}, following the approach presented in~\citep{freni2008optimal} (see also~\citep{bardi1997optimal}). The precise sense of this solution concept is specified in the next definitions.

\begin{definition}
A function $v\in \dC(\R^n_{++})$ is a \emph{viscosity subsolution} of~\eqref{eq:HJB} if, for any $\phi\in \dC^1(\R^n_{++})$ and any local maximum point $k\in\R^n_{++}$ of $v-\phi$,
\[\rho v(k) - \inprod{(L+A)k}{\dD \phi(k)} - H_\gamma(\dD \phi(k))\leq 0.\]
Analogously, $v\in \dC(\R^n_{++})$ is a \emph{viscosity supersolution} of~\eqref{eq:HJB} if, for any $\phi\in \dC^1(\R^n_{++})$ and any local minimum point $k\in\R^n_{++}$ of $v-\phi$,
\[\rho v(k) - \inprod{(L+A)k}{\dD \phi(k)} - H_\gamma(\dD \phi(k))\geq 0.\]
Finally, $v\in \dC(\R^n_{++})$ is a \emph{viscosity solution} of~\eqref{eq:HJB} if it is both a viscosity subsolution and supersolution of~\eqref{eq:HJB}.
\end{definition}

\begin{definition}
A function $v\in \dC(\R^n_{++})$ is a \emph{bilateral viscosity solution} of~\eqref{eq:HJB} if it is a viscosity solution of both~\eqref{eq:HJB} and
\begin{equation}\label{eq:minus_HJB}
-\rho v(k) + \inprod{(L+A)k}{\dD v(k)} + H_\gamma(\dD v(k)) = 0.
\end{equation}
\end{definition}

We also need to introduce the following family of sets, parameterized by $k\in \partial \R^n_+$,
\begin{equation}\label{eq:Uk}
\mathcal U(k)=\{c\in\R^n_+ \colon \exists \, C\in\mathcal A_-(k), \tau>0, \text{ s.t. } C(t) = c, \, \forall t \in (-\tau,0]\}
\end{equation}
and the so-called \emph{inward Hamiltonian}
\begin{equation}
\label{eq:H_in}
H_{\mathrm{in}}(k,q)=\sup_{c\in \mathcal U(k)}\{-\inprod{Nc}{q} + U_\gamma(c)\}, \quad (k,q) \in \partial \R^n_+ \times \R^n.
\end{equation}

\begin{theorem}\label{teo:viscosity_bilateral}
The value function $V$ is a bilateral viscosity solution of~\eqref{eq:HJB} in $\R^n_{++}$, namely,
\begin{equation}\label{eq:visc_subsol}
\rho V(k) = \inprod{(L+A)k}{q} + H_\gamma(q), \quad \forall \, k \in \R^n_{++}, \, q\in \dD^+V(k).
\end{equation}
Moreover,
\begin{equation}\label{eq:visc_boundary}
\rho V(k)-\inprod{(L+A)k}{q} - H_\gamma(q) \leq 0, \quad -\rho V(k)+H_{\mathrm{in}}(k,q)\leq 0, \quad \forall \, k \in \partial \R^n_+, \, q\in \dD^+V(k),
\end{equation}
where $H_{\mathrm{in}}$ is the inward Hamiltonian defined in~\eqref{eq:H_in}.
\end{theorem}

\begin{proof}
We briefly sketch the proof. To prove that $V$ is a bilateral viscosity solution of~\eqref{eq:HJB} in $\R^n_{++}$ it is enough to show that $V$ is a viscosity subsolution of both~\eqref{eq:HJB} and~\eqref{eq:minus_HJB} in $\R^n_{++}$, i.e., that, for any $\phi\in \dC^1(\R^n_{++})$ and any local maximum point $k\in\R^n_{++}$ of $V-\phi$,
\begin{equation}\label{eq:bil_subsol}
\begin{dcases}
\rho V(k) - \inprod{(L+A)k}{\dD \phi(k)} - H_\gamma(\dD \phi(k))\leq 0, \\
-\rho V(k) + \inprod{(L+A)k}{\dD \phi(k)} + H_\gamma(\dD \phi(k))\leq 0.
\end{dcases}
\end{equation}
The subsolution property, that is, \eqref{eq:bil_subsol}, can be proved as in~\citep[Proposition~4.10]{freni2008optimal} and this is equivalent to showing~\eqref{eq:visc_subsol}, see, e.g.,~\citep[Chapter~II, Equation~(1.7)]{bardi1997optimal}.
As a consequence of~\eqref{eq:visc_subsol}, $V$ is also a viscosity supersolution of both~\eqref{eq:HJB} and~\eqref{eq:minus_HJB} in $\R^n_{++}$. Indeed, since $V$ is continuous and concave on $\R^n_{++}$ (see Proposition~\ref{prop_concavity_V}), if $\phi$ is any differentiable function in $\R^n_{++}$ and $k\in\R^n_{++}$ is a local minimum point of $V-\phi$, then $V$ is differentiable at $k$ and $\dD V(k) = \dD \phi(k)$ (cf.~\citep[Chapter~II, Exercise~4.11]{bardi1997optimal}). Therefore,~\eqref{eq:visc_subsol} implies the supersolution property.
Finally,~\eqref{eq:visc_boundary} can be also proved as in~\citep[Proposition~4.10]{freni2008optimal}.
\end{proof}

\begin{theorem}\label{teo: regularity}
The value function $V$ is continuously differentiable in $\R^n_{++}$, $V \in \dC^1(\R^n_{++})$, and it is a classical solution for \eqref{eq:HJB} in $\R^n_{++}$.
%The value function $V$ is continuously differentiable in $\R^n_{++}$, i.e., $V \in \dC^1(\R^n_{++})$.
\end{theorem}

\begin{proof}
Since $V$ is concave, by~\citep[Proposition~3.3.4 (e)]{cannarsasinestrari2004} it is enough to show that $\dD^+V(k)$ is a singleton, for all $k \in \R^n_{++}$. Fix $k \in \R^n_{++}$ and suppose, by contradiction, that there exist $q^1, q^2 \in \dD^+V(k)$, with $q^1 \neq q^2$. Then, since $\dD^+V(k)$ is a convex set (cf.~\citep[Proposition~3.1.5 (b)]{cannarsasinestrari2004}), for each fixed $\lambda \in [0,1]$ we have that $q^\lambda \coloneqq \lambda q^1+(1-\lambda)q^2 \in \dD^+V(k)$. Recalling that, by Proposition~\ref{prop_concavity_V}-\ref{prop:V_superdiff}, $\dD^+V(k) \in \R^n_{++}$ and using that $H_\gamma$, given in~\eqref{eq:hamilton_expl}, is strictly concave in $\R^n_{++}$, we get
\begin{equation*}
H_\gamma(q^\lambda) > \lambda H_\gamma(q^1) + (1-\lambda) H_\gamma(q^2).
\end{equation*}

Using~\eqref{eq:visc_subsol} we get
\begin{equation*}
0 = \rho V(k) - \inprod{(L+A)k}{q^\lambda} - \lambda H_\gamma(q^1) - (1-\lambda) H_\gamma(q^2) > \rho V(k) - \inprod{(L+A)k}{q^\lambda} - H_\gamma(q^\lambda),
\end{equation*}
which contradicts the fact that~\eqref{eq:visc_subsol} holds also for $q^\lambda \in \dD^+V(k)$.
Since $V \in \dC^1(\R^n_{++})$ is a bilateral viscosity solution, we conclude that $V$ is a classical solution for \eqref{eq:HJB} (cf.~\citep[Chapter~II, Proposition~1.3]{bardi1997optimal}).
\end{proof}

\begin{corollary}
\label{cor:feedbackopen}
Fix $k \in \R^n_+$, consider the corresponding optimal control $\widehat C\in \mathcal{A}_+(k)$ for problem~\eqref{eq:ctrlpb_det}, and let $\widehat K^{k, \widehat C}$ be the solution to~\eqref{eq:capitalODE} with initial condition $k$ and control $\widehat C$. Then, for all $t \geq 0$ such that $\widehat K^{k,\widehat C}(t) \in \R^n_{++}$, $\widehat C(t)$ is continuous, strictly positive, and satisfies
\begin{equation}\label{eq:opt_ctrl}
\widehat C(t) = \widehat c(\dD V(K^{k, \widehat C}(t))),
\end{equation}
where $\widehat c$ is the function given in~\eqref{eq:Hamilt_maxim_appl}.
\end{corollary}

\begin{proof}
From Theorem~\ref{teo: regularity}, we know that $V\in \dC^1(\R^n_{++})$. By~\eqref{eq:visc_subsol}, we have that, for all $t \geq 0$ such that $\widehat K^{k,\widehat C}(t) \in \R^n_{++}$,
\begin{equation}\label{regularity_V}
\rho V(\widehat K^{k,\widehat C}(t)) = \inprod{(L+A)\widehat K^{k,\widehat C}(t)}{\dD V(\widehat K^{k,\widehat C}(t))} + H_\gamma(\dD V(\widehat K^{k,\widehat C}(t))).
\end{equation}
Taking the time derivative in~\eqref{DPP2}, we obtain
\begin{equation*}
U_\gamma(\widehat C(t))-\rho  V(\widehat K^{k,\widehat C}(t)) + \inprod{(L+A)K^{k,\widehat C}(t)}{\dD V(\widehat K^{k,\widehat C}(t))} - \inprod{N\widehat C(t)}{\dD V(\widehat K^{k,\widehat C}(t))} = 0.
\end{equation*}
Combining the identity above with~\eqref{regularity_V}, we get that, for all $t \geq 0$ such that $\widehat K^{k,\widehat C}(t) \!\in\! \R^n_{++}$,
\begin{equation*}
U_\gamma(\widehat C(t)) - \inprod{N\widehat C(t)}{\dD V(\widehat K^{k,\widehat C}(t))} = H_\gamma(\dD V(\widehat K^{k,\widehat C}(t))) = \sup_{c \in \R^n_+}\{U_\gamma(c)-\inprod{Nc}{\dD V(\widehat K^{k,\widehat C}(t))}\}.
\end{equation*}
Therefore,~\eqref{eq:opt_ctrl} immediately follows; continuity and strict positivity of $\widehat C$ are easy consequences of the same properties shared with $\widehat c$ and of the continuity of $\dD V$ and of the state trajectories.
\end{proof}

\section{The AK model on networks}\label{sec:AKgraph}
In this section we specify a model for the network of $n \geq 2$ geographically distributed locations introduced in Section~\ref{sec:opt_ctrl}. We show that for this model all the previously introduced assumptions are verified, and hence all results proved in Sections~\ref{sec:ctrl_and_V} and~\ref{sec:DPP_HJB} are satisfied. Moreover, we study classical solutions to the HJB equation~\eqref{eq:HJB} and, by introducing a suitable auxiliary optimization problem, we establish conditions under which we can write explicitly the value function of the optimization problem~\eqref{eq:ctrlpb_det}. Finally, in Section~\ref{sec:twonodes} we discuss some further results in the case where only two geographical locations are involved, i.e., where $n=2$, and in Section~\ref{sec:threenodes} we provide a numerical study of the case where three locations are considered, i.e., $n=3$.

\subsection{The general case}\label{sec:n_nodes}
We model the network of $n \geq 2$ geographically distributed locations by a graph $\rG \coloneqq (\rV,\rE)$, where $\rV$ is a set of vertices, that corresponds to locations, and $\rE$ is a set of edges connecting vertices. We assume that the graph is simple (i.e., not directed), finite, and weighted. We identify $\rV$ with the set $\{1, \dots, n\}$ and $\rE$ as a subset of $\{(i,j) \in \{1, \dots, n\}^2 \text{ s.t. } i \neq j\}$.
%The set of edges can be equivalently described by the adjacency matrix $(a_{ij})_{i,j \in V}$ of the graph, where $a_{ij} = 1$, if $(i,j) \in E$, and $a_{ij} = 0$, otherwise. Notice that, since the graph is simple, the adjacency matrix is symmetric.
We say that two vertices $i,j \in \rV$ are connected, and we write $i \sim j$ if there exist an edge connecting them, i.e.,
\begin{equation*}
i \sim j \quad \Longleftrightarrow \quad (i,j) \in \rE.
\end{equation*}
Finally, we denote by $W = (w_{ij})_{i,j \in \rV}$ the matrix of weigths of the graph, with $w_{ij} \geq 0$, for all $i,j \in \rV$. As we assumed the graph to be simple, $W$ is symmetric. To ease the notation, we assume that vertices $i, j \in \rV$ are not connected if and only if $w_{ij}=0$. In particular, since there are no self-loops, $w_{ii} = 0$, for all $i \in \rV$. Weights can be interpreted as parameters describing the strength of the economic connection between each pair of nodes. They also allow us to deduce the dynamics of the capital flow in the form of a book keeping equation, as follows.
Fix a node $i \in \rV$. We assume that at each time $t > 0$ there is an inflow of capital from each node $j \sim i$, proportional to the capital levels $k_j$ at these connected nodes at time $t$, with proportionality factor $w_{ij}$. The total inflow to site $i$ from connected sites $j$ is, thus,
\begin{equation*}
\sumtwo{j \in \rV}{j \sim i} w_{ij} k_{j} = \sum_{j=1}^{n} w_{ij} k_{j}.
\end{equation*}
Similarly, we assume that at each time $t > 0$ there is an outflow of capital to each node $j \sim i$, proportional to the capital level $k_i$ at time $t$, with proportionality factor $w_{ji} = w_{ij}$. Therefore, the total outflow from site $i$ to connected sites $j$ is
\begin{equation*}
-\sumtwo{j \in \rV}{j \sim i} w_{ij} k_i = -\sum_{j=1}^{n} w_{ij} k_i.
\end{equation*}
Putting together these facts, we can express the dynamics of the capital flow as in~\eqref{eq:capitalODE}, by introducing the symmetric matrix $L = (\ell_{ij})_{i,j \in \rV}$, with
\begin{equation}\label{eq:L}
\ell_{ij} =
\begin{dcases}
w_{ij}, 							&\text{if } i \neq j, \\
-\sumtwo{j \in \rV}{j \neq i} w_{ij},	&\text{if } i = j.
\end{dcases}
\end{equation}
We are going to consider the case where the operator $N$ appearing in~\eqref{eq:capitalODE} is the identity and $\gamma \neq 1$. This choice, together with the fact that $L$ is a Metzler matrix, entails that Assumption~\ref{ass:main} is satisfied. Thus, we obtain the optimization problem
\begin{equation}\label{eq:optpb}
\begin{aligned}
&\sup_{C \in \cA_+(k)} \, \int_0^{+\infty} \de^{-\rho t} \left(\sum_{i=1}^n p_i \dfrac{C_i(t)^{1-\gamma}}{1-\gamma}\right) \dd t \\
&\text{s.t.}
\left\{
\begin{aligned}
&\dfrac{\dd}{\dd t} K(t) = (L+A)K(t) - C(t), \quad t \geq 0,\\
&K(0) = k \in \R^n_{+},
\end{aligned}
\right.
\end{aligned}
\end{equation}
where the set of admissble controls $\cA_+(k)$ is defined in~\eqref{eq:adctrl}.

\begin{remark}\label{rem:laplacian}
Matrix $-L$, where $L$ is defined in~\eqref{eq:L}, corresponds to the standard definition of the graph Laplacian operator on a weighted graph (see, e.g.,  \cite{biyikoglu2007laplacian}).  It is well known that $-L$ generates the positive semidefinite quadratic form
$\langle z, (-Lz) \rangle = \frac{1}{2} \sum_{i=1}^{n} \sum_{j=1}^{n} w_{ij} (z_i -z_j)^2 \geq 0$, $z \in \R^n$. Therefore, $-L$ admits $n$ non-negative eigenvalues, the lowest of which is $0$ with corresponding eigenvector $\bfone = \begin{bmatrix} 1 & \cdots & 1 \end{bmatrix}^T$. Naturally, $L$ is negative semidefinite
with the same eigenvectors as $-L$ and eigenvalues equal in absolute value to those of $L$, but with opposite sign. This fact will be fundamental in what follows.
\end{remark}

To introduce the main assumption of this section, we need the following definitions. To the linear system $(\de^{tL})_{t \geq 0}$ we can associate a graph with $n$ vertices, called \emph{influence graph}, cf.~\citep[Chapter~3]{farinarinaldi2000}. A directed arc $(i,j)$ connects vertices $i$ and $j$ of the influence graph if and only if $\ell_{ji} \neq 0$, i.e., if the $i$-th variable has a direct influence on the $j$-th one.
If the associated influence graph is connected, i.e., if there is a path between any pair of vertices $i$ and $j$, we say that $(\de^{tL})_{t \geq 0}$ is \emph{irreducible}. In other words, irreducibility means that any variable of the linear system $(\de^{tL})_{t \geq 0}$ influences (directly or indirectly) all other variables. Since $L$ is symmetric, the influence graph of the linear system $(\de^{tL})_{t \geq 0}$ coincides with graph $\rG$, except on vertices, because influence graphs may have self-loops. Therefore, to assume that $(\de^{tL})_{t \geq 0}$ is irreducible is equivalent to the fact that $\rG$ is connected, which is in turn equivalent to requiring that the matrix $L$ is itself irreducible.
The next proposition collects some useful results on the matrix $L+A$.
\begin{proposition}\label{prop:L+A}
If the matrix $L$ is irreducible, the following results hold true:
\begin{enumerate}[label=(\roman*)]
\item\label{prop:L+A:positive} The semigroup $(\de^{t(L+A)})_{t \geq 0}$ is positive and irreducible.
\item\label{prop:L+A:symm} Matrix $L+A$ is symmetric, therefore the eigenvalues of $L+A$ are real and there exists an eigenvector orthonormal basis $\{b^0, b^1, \dots, b^{n-1}\}$ of $\R^n$, with associated eigenvalues $\lambda_0 \geq \lambda_1 \geq \cdots \geq \lambda_{n-1}$.
\item\label{prop:L+A:Frobenius} Eigenvector $b^0$ is the unique (up to multiplication by a scalar) strictly positive eigenvector of $L+A$; moreover, the corresponding eigenvalue $\lambda_0$ is simple and satisfies
\begin{equation}\label{eq:lambda0estimate}
0 < \min_{i \in \rV} A_i \leq \lambda_0 \leq \max_{i \in \rV} A_i.
\end{equation}
\end{enumerate}
\end{proposition}
\begin{proof}
Point~\ref{prop:L+A:positive} is immediate, since matrix $A$ is diagonal. Therefore, $L+A$ is a Metzler matrix and the influence graph of $(\de^{t(L+A)})_{t \geq 0}$ is still connected. Clearly, point~\ref{prop:L+A:symm} is due to the symmetry of $L+A$.

Point~\ref{prop:L+A:Frobenius} follows from point~\ref{prop:L+A:positive}. Indeed, we deduce by~\citep[Theorem~17]{farinarinaldi2000} that the eigenvector $b^0$ associated to the highest eigenvalue $\lambda_0$ of $L+A$ is unique (up to multiplication by scalar) and strictly positive, and that $\lambda_0$ is simple. Using the fact that $L+A$ is symmetric and applying~\citep[Theorem~12]{farinarinaldi2000}, we obtain~\eqref{eq:lambda0estimate}.%
\end{proof}
Thus, in the rest of the section, we assume the following
\begin{assumption}\label{ass:irred}
\mbox{}
\begin{enumerate}[label=(\roman*)]
\item\label{ass:L_irr} The matrix $L$ is irreducible.
\item\label{ass:rho2} $\rho>\lambda_0(1-\gamma)$.
\end{enumerate}
\end{assumption}
Notice that Assumption~\ref{ass:irred} ensures that the Assumption~\ref{ass:rho_Frob}  is valid. 
Indeed, symmetry of $L+A$ and Proposition~\ref{prop:L+A}-\ref{prop:L+A:Frobenius} guarantees that Assumption~\ref{ass:rho_Frob}-\ref{ass:Frob} is satisfied, with $\lambda_\star = \lambda_0$ and $b^\star = \kappa b^0$, for some $\kappa > 0$. Assumption~\ref{ass:irred}-\ref{ass:rho2} is a just the same of Assumption~\ref{ass:rho_Frob}-\ref{ass:rho}.
%\begin{remark}
%Symmetry of $L+A$ and Proposition~\ref{prop:L+A}-\ref{prop:L+A:Frobenius} ensure that Assumption~\ref{ass:rho_Frob}-\ref{ass:Frob} is satisfied, with $\lambda_\star = \lambda_0$ and $b^\star = \kappa b^0$, for some $\kappa > 0$. In the following, we will take $b^\star = b^0$.
%We recall that vector $b^0$ is called the \emph{Frobenius eigenvector} and $\lambda_0$ is called the \emph{Frobenius eigenvalue} associated to $L+A$.
%\end{remark}
%To ensure that Assumption~\ref{ass:rho_Frob}-\ref{ass:rho} is verified, we assume that $\rho>\lambda_0(1-\gamma)$ throughout this section. Thus, in the rest of the section, we will assume the following
%
%which garantess Assumption~\ref{ass:rho_Frob} to hold.
In this way, all the results of Sections~\ref{sec:ctrl_and_V} and~\ref{sec:DPP_HJB} hold true. In particular, as established in Theorem~\ref{teo:viscosity_bilateral} the value function $V$ is a bilateral viscosity solution to the HJB equation~\eqref{eq:HJB} associated to the optimal control problem~\eqref{eq:optpb}. Now, we want to analyze classical solutions to~\eqref{eq:optpb} on the strictly positive orthant of $\R^n$ and to provide suitable conditions under which these solutions coincide with the value function $V$. To be precise, we give the following definition.
\begin{definition}\label{def:HJB_class}
A function $v \colon \R^n_{++} \to \R$ is called a \emph{classical solution} to~\eqref{eq:HJB} on $\R^n_{++}$ if $v$ is continuously differentiable in $\R^n_{++}$ and
\begin{equation}\label{eq:HJB_AK}
\rho v(k) = \inprod{(L+A)k}{\dD v(k)} + H_\gamma(\dD v(k)), \quad \forall k \in \R^n_{++},
\end{equation}
where $\gamma > 0$, $\gamma \neq 1$, and $H_\gamma$ is defined in~\eqref{eq:hamilton}.
\end{definition}

We begin with the following result.

\begin{proposition}\label{PROP-AAA} Consider $u \in \R^n_{++}$, $\mu \in \R$, such that $\mu(1-\gamma) > 0$, and define
\begin{equation}\label{ANSATZ-ST}
v(k) \coloneqq \dfrac{1}{1-\gamma} \left[\dfrac{\gamma}{\mu(1-\gamma)}\right]^\gamma \inprod{k}{u}^{1-\gamma}, \quad k \in \R^n_{++}.
\end{equation}
If $L$ is irreducible and if the pair $(\mu, u)$ is a solution to the nonlinear eigenvalue problem
\begin{equation}\label{NONL-EIG-1-ST}
(L+A) u = \left[\frac{\rho}{1-\gamma} - \mu \Phi(u)\right] u,
\end{equation}
where
\begin{equation}\label{CONSTANTS-ST}
\Phi(u)= \sum_{i =1}^{n} p_{i}^{\frac 1\gamma} u_{i}^{\frac{\gamma-1}{\gamma}}, \quad u \in \R^n_{++},
\end{equation}
then $v$ is a classical solution to~\eqref{eq:HJB_AK}, .
\end{proposition}

\begin{proof}
We note, first, that the function $v$ given in~\eqref{ANSATZ-ST} is continuously differentiable in $\R^n_{++}$, with partial derivatives given by
\begin{equation}\label{eq:v_partialderiv}
v_{k_i}(k) = \left[\dfrac{\gamma}{\mu(1-\gamma)}\right]^\gamma \inprod{k}{u}^{-\gamma}u_i, \quad k \in \R^n_{++}, \, i = 1, \dots, n.
\end{equation}
Since $\dD v(k) \in \R^n_{++}$, for all $k \in \R^n_{++}$, we can use~\eqref{eq:hamilton_expl} to rewrite HJB equation~\eqref{eq:HJB_AK} as
\begin{multline*}
\dfrac{\rho }{1-\gamma} \left[\dfrac{\gamma}{\mu(1-\gamma)}\right]^\gamma \inprod{k}{u}^{1-\gamma} =
\left[\dfrac{\gamma}{\mu(1-\gamma)}\right]^\gamma \inprod{k}{u}^{-\gamma} \inprod{(L+A)k}{u}
\\
+ \dfrac{\gamma}{1-\gamma} \left[\dfrac{\gamma}{\mu(1-\gamma)}\right]^{\gamma-1} \inprod{k}{u}^{1-\gamma} \Phi(u), \quad k \in \R^n_{++}.
\end{multline*}
Using the symmetry of $L+A$, the expression above simplifies to
\begin{equation*}
\inprod{k}{(L+A)u} = \left[\dfrac{\rho }{1-\gamma} - \mu \Phi(u)\right] \inprod{k}{u}, \quad k \in \R^n_{++},
\end{equation*}
which is clearly verified if $(\mu, u)$ is a solution to~\eqref{NONL-EIG-1-ST}.
\end{proof}

\begin{remark}
It is interesting to note that~\eqref{NONL-EIG-1-ST} can be seen as a nonlinear Schr\"odinger equation on a graph (see, e.g., \citep{karachalios2005global,sy1992discrete}). Nonlinear eigenvalue problems, like~\eqref{NONL-EIG-1-ST}, are interesting in their own right and can be studied, for instance, using variational techniques based on critical point theory, as in \citep{kravvaritis2020variational}.
\end{remark}

The Frobenius eigenvalue-eigenvector pair $(\lambda_0,b^0)$ of matrix $L+A$, cf. Proposition~\ref{prop:L+A}-\ref{prop:L+A:Frobenius}, provides us with a particular solution to the nonlinear eigenvalue problem~\eqref{NONL-EIG-1-ST}, which is given by
\begin{equation}\label{eq:nonlin_eig_partic}
\mu = \dfrac{1}{\Phi(b^0)}\left(\dfrac{\rho}{1-\gamma} - \lambda_0\right), \qquad u=b^0.
\end{equation}
This observation leads us to study the optimization problem~\eqref{eq:optpb} with the approach of~\citep{CFG2021}, namely, by introducing the half-space $X_{++} \coloneqq \{x \in \R^n \colon \inprod{x}{b^0} > 0\}$ and the auxiliary optimization problem
\begin{equation}\label{eq:Vaux}
V_{b^0}(k) \coloneqq \sup_{C \in \cA_{++}^{b^0}(k)} \int_0^{+\infty} \de^{-\rho t} \left(\sum_{i =1}^n p_i \dfrac{C_i(t)^{1-\gamma}}{1-\gamma}\right) \, \dd t, \quad k \in X_{++},
\end{equation}
where the set of admissible strategies is
\begin{equation*}
\cA_{++}^{b^0}(k) \coloneqq \{C \in \dL^1_{\mathrm{loc}}([0,+\infty); \R^n_+) \, \text{ s.t. }  \inprod{K^{k, C}(t)}{b^0} > 0, \forall t \geq 0\}.
\end{equation*}
The HJB equation associated to the auxiliary problem is
\begin{equation}\label{eq:HJB_aux}
\rho v(k) = \inprod{(L+A)k}{\dD v(k)} + H_\gamma(\dD v(k)), \quad k \in X_{++},
\end{equation}
where the Hamiltonian $H_\gamma$ is given in~\eqref{eq:hamilton}.
We have the following result.

\begin{theorem}\label{th:Vaux}
Suppose that Assumption~\ref{ass:irred} is satisfied. Then,
\begin{enumerate}[label=(\roman*)]
\item The function $v$ defined as
\begin{equation}\label{eq:aux_expl_sol}
v(k) \coloneqq \dfrac{\alpha^\gamma}{1-\gamma} \inprod{k}{b^0}^{1-\gamma}, \quad k \in X_{++},
\end{equation}
where
\begin{equation}\label{eq:alpha}
\alpha = \dfrac{\gamma \Phi(b^0)}{\rho-\lambda_0(1-\gamma)},
\end{equation}
and $\Phi$ is the function given in~\eqref{CONSTANTS-ST}, is a solution to~\eqref{eq:HJB_aux}.
\item For each $k \in X_{++}$, the unique maximizer of $H_\gamma(\dD v(k))$ is given by $\widehat c(\dD v(k)) = Fk$, where $F$ is the $n\times n$ matrix with entries
\begin{equation}\label{eq:Fentries}
f_{ij} \coloneqq \dfrac 1\alpha \left(\dfrac{p_i}{b^0_i}\right)^{\frac 1\gamma} b^0_j, \quad i,j \in \rV.
\end{equation}
\item\label{th:Vaux:geig} The value $g \!\coloneqq\! \frac{\lambda_0-\rho}{\gamma}$ is an eigenvalue of $(L\!+\!A\!-\!F)^T$, with associated eigenvector $b^0$.
\item The \emph{closed loop equation}, i.e., the ODE
\begin{equation}\label{eq:cle}
\begin{dcases}
\dfrac{\dd}{\dd t} K(t) = (L+A-F) K(t), & t \geq 0, \\
K(0) = k,
\end{dcases}
\end{equation}
has a unique solution $\widehat K^{k}$, for each $k \in \R^n$. Moreover, if $k \in X_{++}$, then
\begin{equation*}
\inprod{\widehat K^{k}(t)}{b^0} = \inprod{k}{b^0}\de^{gt}, \quad t \geq 0,
\end{equation*}
where $g$ is given in point~\ref{th:Vaux:geig}, and hence $\widehat K^{k}(t) \in X_{++}$, for all $t \geq 0$.
\item\label{th:Vaux:opt_ctrl} Let $k \in X_{++}$. The control $\widehat C(t)$, $t \geq 0$, with components
\begin{equation}\label{eq:optctrl}
\widehat C_i(t) \coloneqq [F\widehat K^{k}(t)]_i = \dfrac 1\alpha \left(\dfrac{p_i}{b^0_i}\right)^{\frac 1\gamma} \inprod{k}{b^0} \de^{gt}, \quad t \geq 0, \, i \in \rV,
\end{equation}
belongs to $\cA_{++}^{b^0}(k)$ and it is optimal for problem~\eqref{eq:Vaux} starting at $k$. Moreover, the value function $V_{b^0}$ of the auxiliary problem~\eqref{eq:Vaux} verifies $V_{b^0}(k) = v(k)$, where $v$ is the function given in~\eqref{eq:aux_expl_sol}. If, in addition, $\widehat C \in \cA_+(k)$, then $\widehat C$ is optimal for problem~\eqref{eq:optpb} and $V(k) = V_{b^0}(k)$.
\item\label{th:Vaux:V_Vb0} Control $\widehat C$ is optimal for problem~\eqref{eq:optpb} for any $k \in \R^n_+$ if and only if
\begin{equation}\label{cond:V_Vb0}
w_{ij} \geq f_{ij}, \quad \text{for all } i, j \in \rV, \, i \neq j.
\end{equation}
In this case, the value functions $V$ and $V_{b^0}$ of the original and auxiliary problems coincide on $\R^n_+$ and $V_{b^0}$ solves the HJB equation~\eqref{eq:HJB_AK}.
\end{enumerate}
\end{theorem}

\begin{proof}
\mbox{}
\begin{enumerate}[label=(\roman*)]
\item With the same arguments used for the proof of Proposition~\ref{PROP-AAA} and after some simplifications, we can rewrite HJB equation~\eqref{eq:HJB_aux} as
\begin{equation*}
\dfrac{\rho}{1-\gamma} \inprod{k}{b^0} =
\inprod{(L+A)k}{b^0}
+ \dfrac{\gamma}{\alpha(1-\gamma)} \inprod{k}{b^0} \Phi(b^0), \quad k \in X_{++}.
\end{equation*}
Since $(L+A)b^0 = \lambda_0 b^0$, $\inprod{k_0}{b^0} > 0$, and $\alpha > 0$, the equation above simplifies to
\begin{equation*}
\dfrac{\rho}{1-\gamma} =
\lambda_0 + \dfrac{\gamma \Phi(b^0)}{\alpha(1-\gamma)},
\end{equation*}
which yields the value of $\alpha$ given in~\eqref{eq:alpha}. Hence, $v$ is a solution to~\eqref{eq:HJB_aux}.
\item Note that $v$ is continuously differentiable in $X_{++}$ with partial derivatives given by $v_{k_i}(k) = \alpha^\gamma \inprod{k}{b^0}^{-\gamma}b^0_i$, $k \in X_{++}$, $i \in \rV$.
Since $\dD v(k) \in \R^n_{++}$, for all $k \in X_{++}$, we can use~\eqref{eq:Hamilt_maxim_appl} to get that the unique maximizer of $H_\gamma(\dD v(k))$ is given by
\begin{equation*}
\widehat c_i(\dD v(k)) = \left(\dfrac{p_i}{\alpha^\gamma \inprod{k}{b^0}^{-\gamma}b^0_i}\right)^{\frac 1\gamma} = \dfrac 1\alpha \left(\dfrac{p_i}{b^0_i}\right)^{\frac 1\gamma} \inprod{k}{b^0}, \quad i \in \rV, \, k \in X_{++}.
\end{equation*}
It is clear that the expression on the right hand side of the last equation is well-defined for all $k \in \R^n$ and that it is linear. Hence, it can be represented by an $n \times n$ matrix $F$. To determine its entries, we impose, for all $i \in \rV$ and all $k \in \R^n$,
\begin{equation*}
[Fk]_i = \dfrac 1\alpha \left(\dfrac{p_i}{b^0_i}\right)^{\frac 1\gamma} \inprod{k}{b^0} \quad \Longleftrightarrow \quad \sum_{j \in \rV} f_{ij} k_j
= \sum_{j \in \rV} \dfrac 1\alpha \left(\dfrac{p_i}{b^0_i}\right)^{\frac 1\gamma} b^0_j k_j.
\end{equation*}
\item We verify, first, that $b^0$ is an eigenvector of $F^T$. For any $i \in \rV$, we have that
\begin{equation*}
[F^T b^0]_i = \sum_{j \in \rV} \dfrac 1\alpha \left(\dfrac{p_j}{b^0_j}\right)^{\frac 1\gamma} b^0_i b^0_j
= \dfrac{\rho-\lambda_0(1-\gamma)}{\gamma \Phi(b^0)} \Phi(b^0) \, b^0_i
= \dfrac{\rho-\lambda_0(1-\gamma)}{\gamma} b^0_i.
\end{equation*}
Therefore,
\begin{equation*}
(L+A-F)^T b^0 = (L+A)^T b^0 - F^T b^0 = \lambda_0 b^0 - \dfrac{\rho-\lambda_0(1-\gamma)}{\gamma} b^0 = g b^0.
\end{equation*}
\item Standard results ensure uniqueness of the solution $\widehat K^k$ to ODE~\eqref{eq:cle} for any initial condition $k \in \R^n$, which can be written as
\begin{equation*}
\widehat K^k(t) = \de^{t(L+A-F)}k, \quad t \geq 0.
\end{equation*}
Thanks to point~\ref{th:Vaux:geig} proved above, we get
\begin{equation*}
\inprod{\widehat K^k(t)}{b^0} = \inprod{\de^{t(L+A-F)}k}{b^0} = \inprod{k}{(\de^{t(L+A-F)})^T b^0} = \de^{gt} \inprod{k}{b^0}, \quad t \geq 0.
\end{equation*}
Since $k \in X_{++}$, we have that $\inprod{k}{b^0} > 0$, and hence $\widehat K^k(t) \in X_{++}$, for all $t \geq 0$.
\item The solution to~\eqref{eq:capitalODE}, with initial condition $k \in X_{++}$ and control $\widehat C$, coincides with the solution $\widehat K^k$ to~\eqref{eq:cle} by construction. Therefore, since $\widehat K^k(t) \in X_{++}$, for all $t \geq 0$, we have that $\widehat C \in \cA_{++}^{b^0}(k)$. Using a verification theorem analogous to that given in~\citep[Theorem~3.1]{CFG2021}, we get that $V_{b^0}(k) = v(k)$. The last statement of point~\ref{th:Vaux:opt_ctrl} follows from the fact that $\cA_{+}(k) \subseteq \cA_{++}^{b^0}(k)$.
\item Control $\widehat C$ is optimal for problem~\eqref{eq:optpb} for all $k \in \R^n_+$ if and only if $\widehat C \in \cA_+(k)$ for all $k \in \R^n_+$, i.e., if and only if $\widehat K^k(t) \in \R^n_+$, for all $t \geq 0$ and for all $k \in \R^n_+$. In turn, this is equivalent to verifying that the linear system $(\de^{t(L+A-F)})_{t \geq 0}$, providing the unique solution to~\eqref{eq:cle}, is positive, i.e., that $L+A-F$ is a Metzler matrix. This is true if and only if, for all $i,j \in \rV$, $i \neq j$, we have that
\begin{equation*}
[L-A-F]_{ij} \geq 0 \quad \Longleftrightarrow \quad w_{ij} - f_{ij} \geq 0,
\end{equation*}
which gives the stated condition. \qedhere
\end{enumerate}
\end{proof}

\begin{remark}\label{rem:g}
The value $g \coloneqq \frac{\lambda_0-\rho}{\gamma}$ is the optimal growth rate of the economy.
\end{remark}

\begin{remark}\label{rem:invariance}
Condition~\eqref{cond:V_Vb0} is implicit, as the entries of $F$ depend on $\lambda_0$ (through $\alpha$) and $b^0$ which, in turn, depend on the graph weigths $w_{ij}$.
\end{remark}

Theorem~\ref{th:Vaux} shows that the value function $V_{b^0}$ of the auxiliary problem coincides on $\R^n_+$ with the value function $V$ of problem~\eqref{eq:optpb} if only if the graph connectivity is strong enough, i.e., if and only if the weights $w_{ij}$ are big enough, so that condition~\eqref{cond:V_Vb0} holds.
Let us discuss, by contrast, the extreme case where the graph weights are all equal to $0$, i.e., there is no diffusion of capital in the network. This implies that the entries of matrix $L$ verify $\ell_{ij} = 0$, for all $i,j \in \rV$.
Since the objective functional in~\eqref{eq:optpb} is given by the discounted sum of the utility from consumption in each node, which are isolated if all the graphs weights are zero, it is clear that this extreme case is equivalent to solving in each node a standard Ramsey problem.
More precisely, assuming that $\rho > A_{i} (1-\gamma)$, for all $i = 1, \dots, n$, it is not hard to directly show that the value function is
\begin{equation}\label{ANSATZ-UNCOUPLED}
V(k)=\sum_{i=1}^{n} \dfrac{p_i}{1-\gamma}\left[\dfrac{\gamma}{\rho-A_i(1-\gamma)}\right]^\gamma k_{i}^{1-\gamma}, \quad k \in \R^n_+,
\end{equation}
and that the optimal consumption $\widehat C_i$ and the optimal capital path $\widehat K_i^{k, \widehat C}$, for each economy $i = 1, \dots, n$ and for any initial capital endowment $k \in \R^n_+$, is given by
\begin{equation}\label{eq:opt_path_disconn}
\widehat C_i(t) = \dfrac{\rho - A_i(1-\gamma)}{\gamma}k_i \de^{g_i t}, \quad \widehat K_i^{k,\widehat C}(t) = k_i \de^{g_i t},
\end{equation}
where $g_i = \frac{A_i - \rho}{\gamma}$ is the optimal growth rate of each economy.

\begin{remark}\label{rem:non_equiv}
It is important to stress that, except for the cases where condition~\eqref{cond:V_Vb0} or where the graph weigths are all zero, the value function $V$ of problem~\eqref{eq:optpb} is not known in explicit form. These cases are clearly of great interest, both from an economic and from a mathematical point of view. In particular, the cases that are not covered by the results stated in this section are those where the state constraint is binding, and hence a more detailed analysis is necessary. This is left for future research.
\end{remark}

We now want to study the convergence to a steady state of the detrended optimal capital path $\widehat K_g^k$, defined as
\begin{equation}\label{eq:detr_K}
\widehat K_g^k(t) \coloneqq \de^{-gt} \widehat K^k(t), \quad t \geq 0, \, k \in \R^n_+,
\end{equation}
where $\widehat K^k$ is the unique solution to the closed loop equation~\eqref{eq:cle}, with initial condition $k \in \R^n_+$. It is evident that $\widehat K_g^k$ solves the ODE
\begin{equation}\label{eq:cle_detrended}
\begin{dcases}
\dfrac{\dd}{\dd t} K(t) = (B-g) K(t), & t \geq 0, \\
K(0) = k \in \R^n_+,
\end{dcases}
\end{equation}
where $B \coloneqq L+A-F$. The analysis of the asymptotic behaviour of $\widehat K_g^k$ is an important aspect in economics, as it provides us information on the long run behaviour of the economy under study. In particular, an interesting question concerns the existence of steady states. However, it is essential to recall that $\widehat K_g^k$ is the detrended optimal capital path of the auxiliary problem~\eqref{eq:Vaux}, which may not coincide with the detrended optimal capital path of the original problem~\eqref{eq:optpb}. Indeed, we require the latter to be non-negative, i.e., to make sense from an economic point of view.

Clearly, to study the convergence to a steady state of $\widehat K_g^k$ we need to address the question of the stability of the dynamical linear system associated to~\eqref{eq:cle_detrended}. We have the following preliminary result.
\begin{proposition}\label{prop:Bproperties}
Suppose that Assumption~\ref{ass:irred} is verified and that $g > \lambda_1$ (the last two assertions are equivalent to saying that $\lambda_0 > g > \lambda_1$). Let us define the vector
\begin{equation}\label{eq:y}
y \coloneqq \alpha b^0 + \sum_{m=1}^{n-1} \dfrac{\beta_m}{\lambda_m - g} b^m,
	\end{equation}
where $\beta_m \coloneqq \inprod{\widehat c(b^0)}{b^m}$, $m = 1, \dots, n-1$, and $\widehat c$ is the function defined in~\eqref{eq:Hamilt_maxim_appl}.

Then, $\{y, b^1, \dots, b^{n-1}\}$ is a basis for $\R^n$ of eigenvectors of $B$, with corresponding eigenvalues $g, \lambda_1, \dots, \lambda_{n-1}$. In particular, $g$ is the dominant eigenvalue of $B$ and it is simple.
\end{proposition}

\begin{proof}
As an immediate consequence of Theorem~\ref{th:Vaux}-\ref{th:Vaux:geig}, we get that $g$ is an eigenvalue of $B$. Recalling Proposition~\ref{prop:L+A}-\ref{prop:L+A:symm}, we get that, for any $m = 1, \dots, n-1$,
\begin{equation*}
[Fb^m]_i = \dfrac 1\alpha \left(\dfrac{p_i}{b^0_i}\right)^{\frac 1\gamma} \inprod{b^m}{b^0} = 0, \quad i \in \rV,
\end{equation*}
whence we deduce that the remaining eigenvalues of $B$ are $\lambda_1, \dots, \lambda_{n-1}$, with corresponding eigenvectors $b_1, \dots, b_{n-1}$. This fact, together with the assumption $g > \lambda_1$, entails that $g$ is the dominant eigenvalue of $B$ and that it is simple. Finally, a direct computation shows that $y$ is an eigenvector corresponding to $g$. Since $\alpha > 0$, $y$ is linearly independent of $b_1, \dots, b_{n-1}$, and hence $\{y, b^1, \dots, b^{n-1}\}$ is a basis for $\R^n$.
\end{proof}

\begin{theorem}\label{th:stst}
Suppose that Assumption~\ref{ass:irred} is verified and that $g > \lambda_1$. Then, the unique solution $\widehat K^k$ to the closed loop equation~\eqref{eq:cle}, with initial condition $k \in \R^n$, is given by
\begin{equation}\label{eq:cle_expl}
\widehat K^k(t) = \dfrac{\inprod{k}{b^0}}{\alpha} \de^{gt} y + \sum_{m=1}^{n-1} \left(\inprod{k}{b^m} - \dfrac{\beta_m \inprod{k}{b^0}}{\alpha(\lambda_m-g)}\right) \de^{\lambda_m t} b^m, \quad t \geq 0.
\end{equation}
Therefore, for any $k \in \R^n_+$, the detrended optimal capital path $\widehat K_g^k$, defined in~\eqref{eq:detr_K}, verifies
\begin{equation}\label{eq:detr_K_conv}
\widehat K_g^k(t) \longrightarrow \overline K_g^k \coloneqq \dfrac{\inprod{k}{b^0}}{\alpha} y, \quad \text{as } t \to +\infty.
\end{equation}
If, moreover, $w_{ij} \geq f_{ij}$, then $\overline K_g^k \in \R^n_+$, for all $k \in \R^n_+$.
\end{theorem}

\begin{proof}
Thanks to Proposition~\ref{prop:Bproperties}, standard results on linear ODEs ensure that any solution to the closed loop equation~\eqref{eq:cle} can be written as
\begin{equation*}
K^k(t) = c_0 \de^{gt} y + \sum_{m=1}^{n-1} c_m \de^{\lambda_m t} b^m, \quad t \geq 0,
\end{equation*}
for suitable constants $c_0, \dots, c_{n-1} \in \R$. Imposing $K^k(0) = k$ we get~\eqref{eq:cle_expl}. From this equation we also get that
\begin{equation*}
K^k_g(t) = \de^{-gt} K^k(t) = \dfrac{\inprod{k}{b^0}}{\alpha} y + \sum_{m=1}^{n-1} \left(\inprod{k}{b^m} - \dfrac{\beta_m \inprod{k}{b^0}}{\alpha(\lambda_m-g)}\right) \de^{(\lambda_m - g)t} b^m, \quad t \geq 0,
\end{equation*}
whence, given the assumptions, we get~\eqref{eq:detr_K_conv}. The last assertion is a straightforward consequence of Theorem~\ref{th:Vaux}-\ref{th:Vaux:V_Vb0}.
\end{proof}

We conclude this section with a result connected to an interesting question from the economic viewpoint: is it possible to choose the graph $\rG$ so that the growth rate $g$ of optimal consumption path, given in~\eqref{eq:optctrl}, is maximized? As $g = \frac{\lambda_0 - \rho}{\gamma}$, this value depends on the graph structure, that is, its weights, only through the Frobenius eigenvalue $\lambda_0$. Hence, the problem above reduces to determining graph weights such that $\lambda_0$ is maximized.
Clearly (cf. Remark~\ref{rem:non_equiv}), we can interpret $g$ as the growth rate of the economy if condition~\eqref{cond:V_Vb0} is verified, i.e., if the optimal control provided in~\eqref{eq:optctrl} is also optimal for problem~\eqref{eq:optpb}. In this case, the optimal capital path is given by~\eqref{eq:cle_expl} and Theorem~\ref{th:stst} ensures that the growth rates of each node in the economy will converge to $g$ in the long run.

We provide a first answer in this direction studying the case where we optimize $\lambda_0$ by modifying a single graph weight. In what follows we arbitrarily fix $h,\ell \in \rV$, $h \neq \ell$, and we define $\omega \coloneqq w_{h\ell} \geq 0$. We treat $\omega$ as a parameter, while keeping fixed all other graph weights. We denote by $\lambda_0(\omega)$ the Frobenius eigenvalue of $L+A$ corresponding to each possible $\omega \geq 0$.
\begin{proposition}
\label{prop:lambda0dec}
The map $\omega \mapsto \lambda_0(\omega)$ is decreasing on $[0,+\infty)$.
\end{proposition}

\begin{proof}
Let us fix $0 \leq \omega_1 < \omega_2$ and denote by $L_1$ (resp., $L_2$) the matrices with entries as in~\eqref{eq:L} and $w_{h\ell} = \omega_1$ (resp., $w_{h\ell} = \omega_2)$. Then, for each $i,j \in \rV$,
\begin{equation*}
(L_2 - L_1)_{ij} =
\begin{dcases}
0,	&\text{if } i \neq h, j \neq \ell, \text{ or } i=h, j \neq \ell, \text{ or } j=h, i \neq \ell, \\
\omega_2 - \omega_1, &\text{if } i = h, j = \ell, \\
\omega_1 - \omega_2, &\text{if } i=j=h, \text{ or } i=j=\ell.
\end{dcases}
\end{equation*}
It is not difficult to show that $P(\lambda) = (-1)^n \lambda^{n-1}[\lambda - 2(\omega_1 - \omega_2)]$ is the characteristic polynomial of $L_2 - L_1$, whence we deduce that $L_2 - L_1$ has only two distinct eigenvalues: zero and $2(\omega_1 - \omega_2) < 0$. Therefore, $L_2 - L_1$ is negative semidefinite, i.e.,
$z^T L_2 z \leq z^T L_1 z$, for all $z \in \R^n$. This also implies $z^T (L_2 + A) z \leq z^T (L_1 + A) z$, for all $z \in \R^n$.

Let $b^0(\omega_2)$ be the unit Frobenius eigenvector of $L_2 + A$. Then, by the Rayleigh's principle
\begin{align*}
\lambda_0(\omega_2)
&= \maxtwo{z \in \R^n}{\norm{z}=1} z^T (L_2 + A) z = b^0(\omega_2)^T (L_2 + A) b^0(\omega_2) \leq b^0(\omega_2)^T (L_1 + A) b^0(\omega_2)
\\
&\leq \maxtwo{z \in \R^n}{\norm{z}=1} z^T (L_1 + A) z = \lambda_0(\omega_1),
\end{align*}
whence the claim.
\end{proof}

This result implies that, in the case where it is possible to tweak only one graph weight, one should choose the lowest possible weight guaranteeing that~\eqref{cond:V_Vb0} is verified, in order to maximize the growth rate of the economy $g$, i.e., the asymptotic growth rate of each node in the economy. This suggests that connections between vertices that are too strong create a congestion effect, which is negative from the economic viewpoint.

\subsection{The two-nodes case}\label{sec:twonodes}
In this section we further analyze Problem~\eqref{eq:optpb} in the case where the graph $\rG$ introduced at the beginning of Section~\ref{sec:AKgraph} has only two nodes, i.e., $n = 2$. Without loss of generality we assume that the technological levels $A_1$ and $A_2$ verify $A_1 \leq A_2$ (if this is not the case, it is enough to re-label the two nodes of the graph). For the sake of simplicity, we also assume that $p_1 = p_2 = 1$, where $p_1, p_2$ are the parameters appearing in the gain functional given in~\eqref{eq:optpb}. In this case, there is only one graph weight, denoted by $w \geq 0$, and the matrix $L+A$ reads
\[L+A=\begin{bmatrix}
A_1-w &w\\
w& A_2-w
\end{bmatrix}.\]
By some straightforward calculations, we can explicitly compute the Frobenius eigenvalue and eigenvector as a function of the graph weight $w$:
\begin{align}
\lambda_0(w)&=
\begin{cases}
A, &\text{if } A_1=A_2=A, \\
-w+\frac{A_1+A_2}{2}+\sqrt{w^2+\frac{(A_1-A_2)^2}{4}}, &\text{if } A_1\neq A_2,
\end{cases} \label{eq:lambda_0}
\\
b^0(w)&=
\begin{cases}
(1,1), &\text{if } A_1=A_2=A,\\
(b^0_1(w), b^0_2(w)), &\text{if } A_1\neq A_2,
\end{cases} \label{eq:b_0}
\end{align}
with the ratio between the two components of $b^0$ verifying $\frac{b^0_2(w)}{b^0_1(w)}=\frac{A_2-A_1}{A_2-\lambda_0(w)}-1$. It is worth noting that, as a consequence of Proposition~\ref{prop:lambda0dec}, we have that, for all $w \geq 0$,
\begin{equation}\label{eq:lambda0est}
\lim_{w \to +\infty} \lambda_0(w) = \dfrac{A_1 + A_2}{2} < \lambda_0(w) \leq A_2 = \lambda_0(0).
\end{equation}

In the two-nodes case we can study the connection between the value functions $V$ and $V_{b^0}$ of optimization problem~\eqref{eq:optpb} and of the auxiliary problem~\eqref{eq:Vaux} in a more precise way, as Proposition~\ref{prop:V_Vb0_equiv} below shows. In the following, to stress the dependence of $V$ and $V_{b^0}$ on the parameter $w$, we write $V(w;k)$ and $V_{b^0}(w;k)$, for each fixed $k \in \R^n_{++}$. To ensure that Theorem~\ref{th:Vaux} can be applied, we assume that
\begin{equation}\label{eq:twonodes_cond}
w > 0, \qquad \rho > A_2(1-\gamma).
\end{equation}
Condition~\eqref{eq:twonodes_cond} is equivalent to say that Assumption~\ref{ass:irred} is verified for all $w > 0$.

\begin{proposition}\label{prop:V_Vb0_equiv}
Suppose that the graph $\rG$ has only two nodes, i.e., $n=2$, and that $p_1 = p_2 = 1$. Fix $k \in \R^n_{++}$.
\begin{enumerate}[label=(\roman*)]
\item\label{prop:twonodes_A} In the case where $A_1=A_2=A$, $V(w;k)=V_{b^0}(w;k)$ if and only if $w \geq \frac{\rho-A(1-\gamma)}{2\gamma}$.
\item\label{prop:twonodes_A1} In the case where $A_1 < A_2$, there exist $\overline w > \underline w > 0$ such that $V(w;k)=V_{b^0}(w;k)$, for all $w \in [\overline w, +\infty)$, and $V(w;k) \neq V_{b^0}(w;k)$, for all $w \in (0,\underline w)$.
\end{enumerate}
\end{proposition}

\begin{proof}
Case~\ref{prop:twonodes_A} is a straightforward consequence of condition~\eqref{cond:V_Vb0} and of the fact that $f_{12}(w) = f_{21}(w) = \frac{\rho-A(1-\gamma)}{2\gamma}$, for all $w \geq 0$, where $f_{12}$ and $f_{21}$ are the off-diagonal elements of matrix $F$, introduced in~\eqref{eq:Fentries}, which depend on the graph weight $w$.

To prove the first part of the statement of case~\ref{prop:twonodes_A1}, we show first that there exists $\widetilde w > 0$ such that the system of inequalities
\[
\begin{cases}
f_{12}(w)\leq \frac{\rho-A_1(1-\gamma)}{2\gamma}\\
f_{21}(w)\leq \frac{\rho-A_1(1-\gamma)}{2\gamma}
\end{cases}
\]
is satisfied for all $w\in (\widetilde w, \infty)$.
The two inequalities above can be rearranged as
\begin{equation}\label{cond:f12_f21_bound}
\begin{cases}
\Psi(w)\leq \frac12 \left(G(w)^{-1}+G(w)^{-\frac{1}{\gamma}}\right)\\
\Psi(w)\leq \frac12 \left(G(w)+G(w)^{\frac{1}{\gamma}}\right)\\
\end{cases}
\end{equation}
where, for all $w > 0$,
\begin{equation}\label{eq:G}
G(w) \coloneqq \frac{b^0_2(w)}{b^0_1(w)}=\frac{A_2-A_1}{A_2-\lambda_0(w)}-1, \qquad \Psi(w) \coloneqq \frac{\rho-\lambda_0(w)}{\rho-A_1(1-\gamma)}.
\end{equation}
%We observe that $\lambda_0$ is monotonically decreasing and
%\begin{equation}\label{eq:limits_lambda}
%\lim_{w\to 0} \lambda_1(w)=A_2, \lim_{w\to \infty} \lambda_1(w)=\frac{A_1+A_2}{2},
%\end{equation}
%while $G$ is monotonically decreasing and
%\[\lim_{w\to 0} G(w)=+\infty, \lim_{w\to \infty} G(w)=1. \]
With routine computations it is possible to show that $\Psi$ is monotonically decreasing in $(0,+\infty)$ and
\[
\lim_{w\to 0} \Psi(w)=1,
\qquad 
\lim_{w\to +\infty} \Psi(w)=\frac{2\rho-(A_1+A_2)}{2[\rho-A_1(1-\gamma)]}<1, 
\]
while the functions $w \mapsto \frac12\left(G(w)+G(w)^{\frac{1}{\gamma}}\right)$ and $w \mapsto \frac12\left(G(w)^{-1}+G(w)^{-\frac{1}{\gamma}}\right)$ are, respectively, monotonically decreasing and increasing in $(0,+\infty)$, with
\begin{gather}
\lim_{w\to 0} \frac12\left(G(w)+G(w)^{\frac{1}{\gamma}}\right)=+\infty, \qquad \lim_{w\to +\infty} \frac12\left(G(w)+G(w)^{\frac{1}{\gamma}}\right)=1, \label{eq:limits_G}
\\
\lim_{w\to 0} \frac12\left(G(w)^{-1}+G(w)^{-\frac{1}{\gamma}}\right)=0, \qquad \lim_{w\to +\infty} \frac12\left(G(w)^{-1}+G(w)^{-\frac{1}{\gamma}}\right)=1. 
\label{eq:limits_G_-}
\end{gather}
We conclude that the first inequality in~\eqref{cond:f12_f21_bound} is verified for all $w > 0$ and that the second one is satisfied for all $w \geq \widetilde w$, for some $\widetilde{w} > 0$. 
Therefore, if $w \geq \frac{\rho-A_1(1-\gamma)}{2\gamma}$ and $w \geq \widetilde w$, then condition~\eqref{cond:V_Vb0} is satisfied. Hence, we deduce that $V(w)=V_{b_0}(w)$, for all $w \geq \overline w \coloneqq \max\left\{\frac{\rho-A_1(1-\gamma)}{2\gamma}, \widetilde w\right\}$.

The entry $f_{12}$ of matrix $F$ can be rewritten as
\begin{equation*}
f_{12}(w) = \dfrac{\rho - \lambda_0(w)(1-\gamma)}{\gamma \left[G(w)^{-1} + G(w)^{-\frac 1\gamma}\right]}, \quad w > 0.
\end{equation*}
It can be proved \emph{via} standard computations that the function $w \mapsto f_{12}(w)$ is monotonically decreasing in $(0,+\infty)$; moreover, using Proposition~\ref{prop:lambda0dec}, Equation~\eqref{eq:lambda_0}, and the properties of the map $w \mapsto G(w)^{-1} + G(w)^{-\frac 1\gamma}$ established above, we have that
\begin{equation*}
\lim_{w \to 0} f_{12}(w) = +\infty.
\end{equation*} 
Therefore, the second part of the statement of case~\ref{prop:twonodes_A1} follows, noting that the inequality $f_{12}(w) \leq w$ is not verified for $w < \underline w$, for some $0 < \underline w < \overline w$.
\end{proof}

Our final result establishes some monotonicity properties of the value function $V$ of optimization problem~\eqref{eq:optpb} with respect to the graph weight $w$. We denote by $\overline A \coloneqq \frac{A_1 + A_2}{2}$ and $\overline k \coloneqq \frac{k_1+k_2}{2}$ the averages of the technological levels and of the initial capital endowments, respectively, and by $g_\infty \coloneqq \frac{\overline A (1-\gamma) - \rho}{\gamma}$. Note that $g_\infty < 0$, by virtue of condition~\eqref{eq:twonodes_cond}, and that $g_\infty = \lim\limits_{w \to +\infty} g(w) = \frac{\lambda_0(w)-\rho}{\gamma}$, where $g(w)$ is the optimal growth rate of the economy, defined in Remark~\ref{rem:g}, as a function of the graph weight $w > 0$.

\begin{proposition}\label{prop:max_V_2_node}
Suppose that the graph $\rG$ has only two nodes, i.e., $n=2$, and that $p_1 = p_2 = 1$. Fix $k \in \R^n_{++}$.
\begin{enumerate}[label=\arabic*.]
\item\label{prop:max_V_2_node:Aeq} In the case where $A_1 = A_2 = A$, the map $w \mapsto V(w;k)$ is constant on $[\frac{\rho-A(1-\gamma)}{2\gamma},+\infty)$;
\item\label{prop:max_V_2_node:Aneq} In the case where $A_1 < A_2$, the map $w \mapsto V(w;k)$ is
\begin{enumerate}[label=(\roman*)]
\item\label{prop:max_V_2_node:k1_<_k2} Monotonically decreasing on $[\overline w, +\infty)$ if $k_1 \leq k_2$, where $\overline w > 0$ is the constant appearing in Proposition~\ref{prop:V_Vb0_equiv};
\item\label{prop:max_V_2_node:k1_>_k2_dec} Monotonically decreasing on $[w_\star, +\infty)$, for some $w_\star \!>\! 0$, if $k_1 \!>\! k_2$ and $\frac{A_2 - A_1}{g_\infty} \!>\! \frac{k_2 - k_1}{\overline k}$;
\item\label{prop:max_V_2_node:k1_>_k2_inc} Monotonically increasing on $[w^\star, +\infty)$, for some $w^\star \!>\! 0$, if $k_1 \!>\! k_2$ and $\frac{A_2 - A_1}{g_\infty} \!<\! \frac{k_2 - k_1}{\overline k}$.
\end{enumerate}
\end{enumerate}
\end{proposition}

\begin{proof}
See Appendix~\ref{app:V2nodes}.
\end{proof}

%{\color{red}
%\begin{remark}
%This result can be generalized for a general dimension, $n$. Indeed, let us consider the homogeneous case, $w_{ij}=w$ and the vector representing the population level at each site is $P=I$. Then, by rearranging the expression given by formula \eqref{eq:value_function_aux_prob}
%\[V_{b^0}(k_0)=\frac{1}{1-\gamma}\left(\frac{\gamma}{\rho-\lambda_0(w)(1-\gamma)}\right)^\gamma\cdot\left(\sum_{m\in V}(b^0_m)^{\frac{\gamma-1}{\gamma}}\right)^\gamma\cdot \langle k_0,b^0\rangle^{1-\gamma}.\]
%
%Both the first and the second term of the product reaches their maximum at $w=0$. The maximum of the third term depends on the initial condition $k_0$. We do not present a complete study about that, we just conclude that, as presented in Proposition \ref{prop:max_V_2_node}, the maximum depends on the initial condition $k_0$.
%\end{remark}
%}

\subsection{The three-nodes case: a numerical study}\label{sec:threenodes}
To make clear the results
of Section~\ref{sec:AKgraph} we provide a three-region numerical example. To calibrate the model, we choose the discount rate $\rho=0.03$ and parameters $\gamma$ and $p_i$, appearing in~\eqref{eq:Jfunct}, as follows: $\gamma = 3$, $p_{i}=\frac 13$, for $i=1,2,3$. We assume\footnote{The values chosen for $\rho$ and $\gamma$ are in line with economic theory. The values used for $A_1$, $A_2$, and $A_3$ are such that, in combination with $\rho$ and $\gamma$, they result in realistic values for growth rates in the model where the three regions are isolated, which we use as a benchmark.} that the initial capital endowments are $k_i=1$, for $i = 1,2,3$.
Let the marginal productivity of capital in each region be $\left(A_{1},A_{2},A_{3}\right)  =\left(0.10,0.12,0.08\right)$. Note that these parameters satisfy the condition $\rho > A_i(1-\gamma)$, for $i = 1,2,3$.

If these three regions are isolated, i.e., the graph weights are all equal to zero, then each of the optimal capital stocks $\widehat K_i$, production outputs $A_i \widehat K_i$, and optimal consumptions $\widehat C_i$, for $i=1,2,3$, respectively grow at the constant rates $\bar{g}_{i}=\frac{A_{i}-\rho}{\gamma}$, for $i=1,2,3$, as can be seen from~\eqref{eq:opt_path_disconn}.
With the choice of parameters above, we have 
$\left(\bar{g}_{1},\bar{g}_{2},\bar{g}_{3}\right) \simeq \left(0.0233,0.03,0.0166\right)$. 
We also note that in this case there is no transition dynamics.
This is in accordance with the standard literature (see, e.g.,~\citep[Chapter~4]{barro2004economic}).

We introduce symmetric connections between the regions, by specifying the graph weights
$\left(w_{12},w_{13},w_{23}\right) = \left(0.04,0.03,0.05\right)$ and setting $w_{ij} = w_{ji}$, for $i,j=1,2,3$, $i \neq j$. With this choice, we have that the Frobenius eigenvalue $\lambda_0$ and the Frobenius eigenvector $b^0$ of matrix $L+A$ are $\lambda_0 \simeq 0.1019$, $b^0 \simeq \left(0.5707, 0.6584, 0.4908\right)$, the optimal growth rate of the economy (cf. Remark~\ref{rem:g}) is $g =\frac{\lambda_{0}-\rho}{\gamma} \simeq 0.02398$, and $\lambda_1 \simeq -0.005$. Hence, all the assumptions of Theorem~\ref{th:Vaux} and Theorem~\ref{th:stst} are verified. Moreover, we can numerically verify that condition~\eqref{cond:V_Vb0} is satisfied, and thus we can assert that the optimal paths for the regional capital stocks $\widehat K_i$ and the optimal consumption paths $\widehat C_i$, for $i =1,2,3$, are those given by~\eqref{eq:cle_expl} and~\eqref{eq:optctrl}, respectively.
We can also calculate the optimal paths for the regional
growth rates, defined as $g_{i}(t) \coloneqq \frac{1}{\widehat{K}_i(t)}\frac{\dd\widehat{K}_i(t)}{\dd t}$. 
Capital stocks grow from the their initial values, while the regional
growth rates start from different initial values and converge to
the common growth rate $g$, as
predicted by the results in Section~\ref{sec:AKgraph}. As shown in Figure~\ref{fig:3nodes}\subref{fig:growth} the
model with interconnected regions acquires transition dynamics, since the regional growth rates $g_i$
converge to the common long-run growth rate $g$. In particular, we can verify numerically that all three growth rates reach the interval $(0.99g, 1.01g)$, i.e., $\pm 1\%$ the value of $g$, at $t \simeq 81.72$.

Aggregate discounted utility~\eqref{eq:Jfunct} computed at the optimal consumption
path given by~\eqref{eq:optctrl} is numerically confirmed to be finite, as predicted by theory (cf. Proposition~\ref{prop: existence_optimal_control}), and coincides with the value given in~\eqref{eq:aux_expl_sol}; the verification of this fact is important from the economic viewpoint.
%{\color{red}(AC: I am not sure whether to keep or not the last sentence. Since condition~\eqref{cond:V_Vb0} is verified, by Theorem~\ref{th:Vaux}-\ref{th:Vaux:V_Vb0} we know that the value function is explicitly given by~\eqref{eq:aux_expl_sol}. So it seems to me that there is no need to numerically integrate the aggregate discounted utility to compute the optimal welfare, but it is enough to directly compute the value function, which in this case has an explicit formula)}

In terms of policy insights the results suggest that a
social planner by connecting regions can increase the long run growth rates of
regions with low marginal productivity of capital. The trade-off is that the
long run growth rates of the high productivity region will be reduced by this connection.
Our numerical simulations also suggest that an increase in
the connectivity, which is reflected in the graph weights, will reduce both the common growth rate $g$ at which the regions converge and the convergence time, understood as the time at which the three growth rates reach a sufficiently small neighborhood of the long-run growth rate $g$, e.g., $\pm 1\%$ the value of $g$, as above. In particular, if the planner increases the connectivity only between a pair of regions while keeping the other ones fixed, the reduction of $g$ is a consequence of Proposition~\ref{prop:lambda0dec}.

\begin{figure}
	\centering
	\subfloat[][Regional growth rates.]
		{\includegraphics[scale = 0.185]{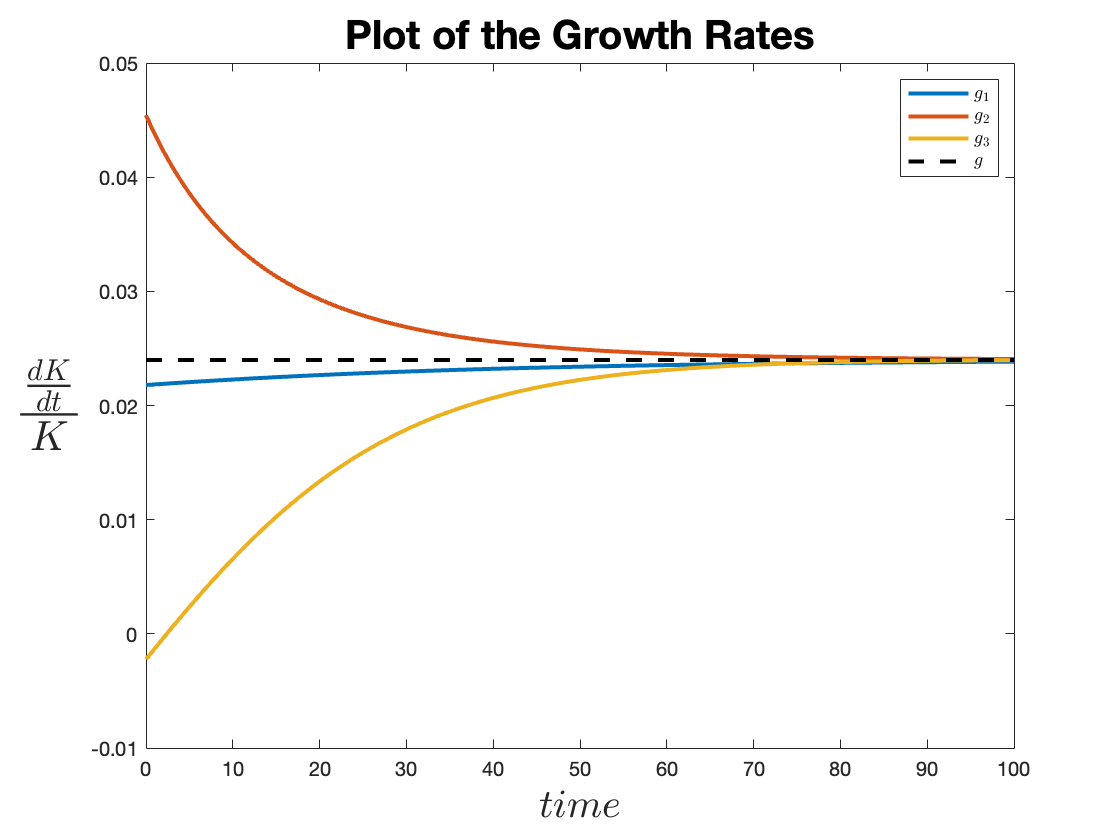}\label{fig:growth}} \quad
	\subfloat[][Time $T_-$.]
		{\includegraphics[scale = 0.185]{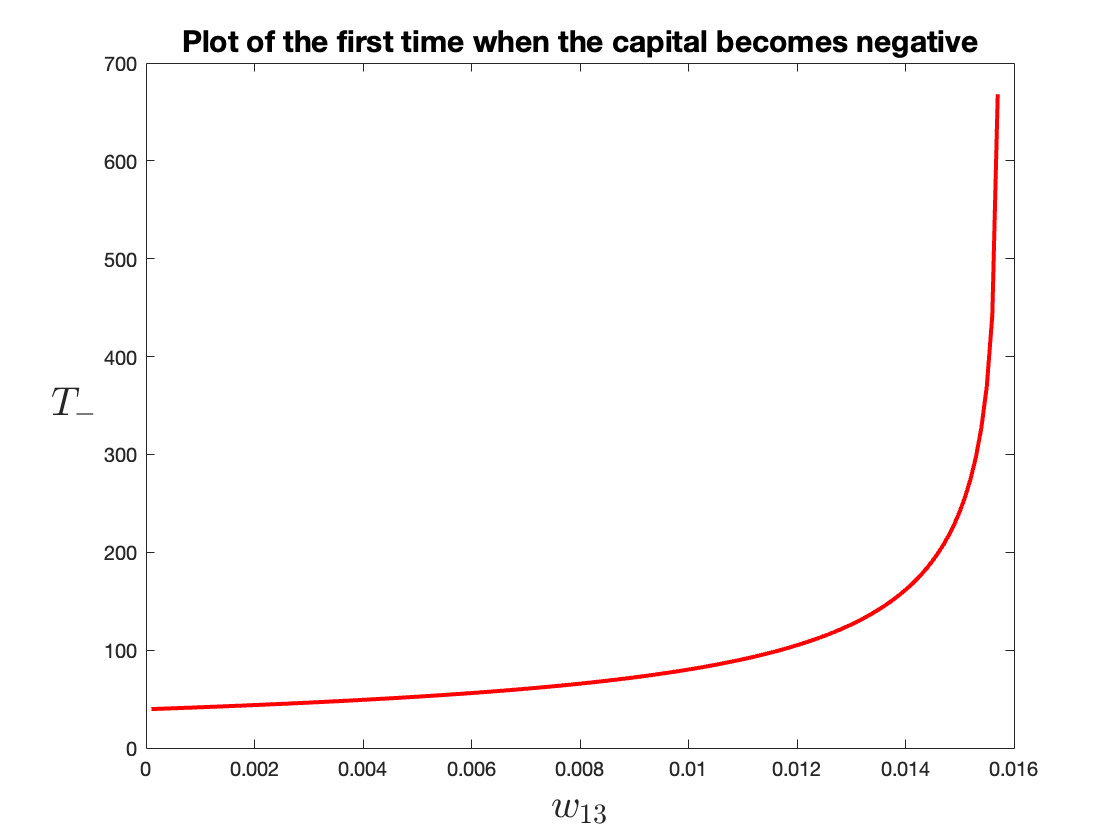}\label{fig:neg_cap}} \quad
	\caption{\protect\subref{fig:growth}: Convergence of the regional growth rates; \protect\subref{fig:neg_cap}: time $T_-$ at which at least one of the capital stocks goes negative.}
	\label{fig:3nodes}
\end{figure}

By Theorem~\ref{th:Vaux}-\ref{th:Vaux:opt_ctrl}, the control $\widehat C$ provided in~\eqref{eq:optctrl} is optimal for problem~\eqref{eq:Vaux}, i.e., for the auxiliary problem, but it may not be admissible for the original problem~\eqref{eq:optpb}. Indeed, it may happen that the optimal trajectory $\widehat K$ given in~\eqref{eq:cle_expl} exits the positive orthant, i.e., that at least one capital path becomes negative. If this is the case, we speak of \emph{solution breakdown}, to highlight the fact that the solutions of the auxiliary and original problems do not coincide. This happens if the values of the
connection parameters, i.e., the graph weights, fall below a certain threshold. 
To explore this issue numerically, we evaluate the evolution of the capital path $\widehat K$ by keeping parameters $w_{12}$ and $w_{23}$ fixed as specified above, and
varying $w_{13}$. In particular, for each $w_{13}$ in a fixed grid of possible values, we compute the time of solution breakdown, that is, the quantity
\begin{equation*}
T_- \coloneqq \inf\{t > 0 \colon \widehat K(t) \notin \R^n_+\},
\end{equation*}
under the usual convention $\inf \emptyset = +\infty$.
Figure~\ref{fig:3nodes}\subref{fig:neg_cap} shows a graph of $T_-$ vs. $w_{13}$. Clearly, the lower the value of $w_{13}$ the earlier we have a solution breakdown. Note that $T_- \to +\infty$ as $w_{13}$ approaches a threshold $w_{13}^\star$, which in this case is slightly lower than $0.016$. This highlights the fact that for all values of $w_{13}$ higher than $w_{13}^\star$ the capital path $\widehat K$ stays in the positive orthant, which implies that $\widehat C$ is admissible (and hence optimal) for the original problem starting at $k$, as stated in Theorem~\ref{th:Vaux}-\ref{th:Vaux:opt_ctrl}. Therefore, our study provides us with a numerical method to evaluate, for each fixed initial capital endowment, the thresholds for the graph connectivities above which the optimal consumption plans and the optimal capital paths are the ones given in~\eqref{eq:optctrl} and~\eqref{eq:cle_expl}, respectively. Otherwise said, we are verifying \emph{a posteriori} that the solutions of the original and the auxiliary problems coincide at $k$. 
Finally, note that the value of $w_{13}^\star$ is lower than the threshold $\overline w_{13}$ above which condition~\eqref{cond:V_Vb0} is verified, which in this case is $\overline w_{13} \simeq 0.0274$. Indeed, under condition~\eqref{cond:V_Vb0}, that is, for $w_{13} \geq \overline w_{13}$, we have that the solutions of auxiliary and the original problems coincide for all $k \in \R^n_+$; instead, for $w_{13}^\star \leq w_{13} \leq \overline w_{13}$, these two solutions coincide at the initial capital endowment $k$ given above, but not necessarily for other values of $k$. For instance, setting $w_{13} = 0.02$, $k = \left(1, 0.1, 0.05\right)$, and the remaining parameters fixed as above, we verify numerically that $T_- \simeq 9.0295$, and hence we have a solution breakdown for this particular set of parameters.

\appendix
\section{Technical results}\label{app:technical}
We provide here some technical results used in Section~\ref{sec:ctrl_and_V}.

\begin{lemma}\label{lem:bound_c}
Suppose that Assumption~\ref{ass:rho_Frob} holds and let $C$ be an admissible control for $k\in \R^n_+$, i.e., $C \in \mathcal{A}_+(k)$. Then, there exists $\Gamma > 0$, independent of $k$ and $C$, such that
\begin{equation}\label{eq:_bound_c}
\int_0^{+\infty} \de^{-s\lambda_\star} \norm{C(s)} \, \dd s \leq \Gamma \inprod{k}{b^\star}.
\end{equation}
\end{lemma}

\begin{proof}
Let $K^{k,C}$ be the unique solution to~\eqref{eq:capitalODE} with initial condition $k \in \R^n_+$ and control $C \in \mathcal{A}_+(k)$. Then, $K^{k,C}(t) \in \R^n_+$, for all $t \geq 0$. Taking the inner product with respect to $b^\star$ in~\eqref{eq:solcapitalODE}, we get
\begin{align*}
0 \leq \langle K^{k,C}(t), b^\star\rangle&=\langle \de^{t(L+A)}k,b^\star\rangle-\int_0^t \langle \de^{(t-s)(L+A)}NC(s),b^\star\rangle \, \dd s\\
&=\de^{t\lambda_\star}\left[\langle k,b^\star\rangle-\int_0^t \de^{-s\lambda_\star} \langle C(s), N^T b^\star\rangle \, \dd s\right], \quad t \geq 0.\\
\end{align*}
Thus, for all $t \geq 0$,
\begin{equation*}
\int_0^t \de^{-s\lambda_\star} \langle C(s), N^T b^\star\rangle \, \dd s \leq \langle k,b^\star\rangle \, .
\end{equation*}
Thanks to Assumption~\ref{ass:main}-\ref{ass:Npos}, we have that $N^T b^\star \in \R^n_{++}$. Hence, recalling that $C(s) \in \R^n_+$, for all $s \geq 0$, and setting $\chi \coloneqq \min\limits_{i = 1, \dots, n} [N^T b^\star]_i > 0$, we have that
\begin{equation*}
\inprod{C(s)}{N^T b^\star} \geq \chi \sum_{i=1}^n \abs{C_i(s)} \geq \chi \norm{C(s)}, \quad s \geq 0.
\end{equation*}
Therefore, for all $t \geq 0$,
\begin{equation*}
\int_0^t \de^{-s\lambda_\star} \norm{C(s)} \, \dd s \leq \dfrac{1}{\chi} \int_0^t \de^{-s\lambda_\star} \langle C(s), N^T b^\star\rangle \, \dd s \leq \dfrac{1}{\chi} \langle k,b^\star\rangle,
\end{equation*}
whence we deduce~\eqref{eq:_bound_c}.
\end{proof}

We need to introduce the following spaces of functions, defined for all $p \geq 1$ as
\begin{equation}\label{eq:Lp}
\dL^p_{\lambda_\star}([0,+\infty);\R^n) \coloneqq \left\{f \colon [0,+\infty) \to \R^n \text{ s.t. } \int_0^{+\infty} \de^{-t\lambda_\star} \norm{f(t)}^p \, \dd t < +\infty\right\},
\end{equation}
where $\lambda_\star \in \R$ is the Frobenius eigenvalue of $(L+A)^T$, equipped with the norm
\begin{equation*}
\norm{f}_{p, \lambda_\star} \coloneqq \left(\int_0^{+\infty} \de^{-t\lambda_\star} \norm{f(t)}^p \, \dd t\right)^{\frac 1p}.
\end{equation*}

\begin{lemma}\label{lem:bound_c_cons}
Let $\gamma \in (0,1)$ and consider a family $(C_m)_{m \geq 1} \subset \dL^1_{\mathrm{loc}}([0,+\infty); \R^n_+)$, such that
\begin{equation}\label{eq:Cm_bound}
\int_0^{+\infty} \de^{-s\lambda_\star} \norm{C_m(s)} \, \dd s \leq \Psi,
\end{equation}
where $\Psi > 0$ is a constant independent of $(C_m)_{m \geq 1}$.

Then, there exists $C_0\in \dL^1_{\lambda_\star}([0,+\infty); \R^n)$, with $C_0(t) \in \R^n_+$, for all $t \geq 0$, such that, up to subsequences,
\begin{equation}\label{eq:_bound_c_n_cons1}
C_m^{1-\gamma}\rightarrow C_0^{1-\gamma}, \quad \text{weakly in}\,\, \dL^{\frac{1}{1-\gamma}}_{\lambda_\star}([0,+\infty); \R^n), \text{ as } m \to \infty.
\end{equation}
Moreover, for all $w \in \R^n$ and all $t > 0$,
\begin{equation}\label{eq:_bound_c_n_cons2}
\int_0^t \inprod{\de^{(t-s)(L+A)} NC_0(s)}{w} \dd s \leq \liminf_{m \to \infty} \int_0^t \inprod{\de^{(t-s)(L+A)} NC_m(s)}{w} \dd s,
\end{equation}
and, if Assumption~\ref{ass:rho_Frob}-\ref{ass:rho} is verified (recalling that $J$ is the functional defined in~\eqref{eq:Jfunct}),
\begin{equation}\label{eq:_bound_c_n_cons3}
\lim_{m\to \infty} J(C_m)=J(C_0).
\end{equation}
\end{lemma}

\begin{proof}
Fix $\gamma \in (0,1)$. By Jensen's inequality, $\norm{x^{1-\gamma}} \leq n^{\frac{\gamma}{2}} \norm{x}^{1-\gamma}$, for any $x \in \R^n_+$, where $x^{1-\gamma}$ is the vector whose components are $x_i^{1-\gamma}$, $i = 1, \dots, n$. Therefore, we deduce from~\eqref{eq:Cm_bound}
\begin{equation*}
\int_0^{+\infty} \de^{-s\lambda_\star} \norm{C_m(s)^{1-\gamma}}^{\frac{1}{1-\gamma}} \, \dd s \leq n^{\frac{\gamma}{2(1-\gamma)}} \Psi < +\infty.
\end{equation*}
This implies that $C_m^{1-\gamma} \in  \dL^{\frac{1}{1-\gamma}}_{\lambda_\star}([0,+\infty); \R^n)$, for all $m \geq 1$, and that $\norm{C_m^{1-\gamma}}_{\frac{1}{1-\gamma},\lambda_\star}$ is uniformly bounded. Since, $\dL^{\frac{1}{1-\gamma}}_{\lambda_\star}([0,+\infty); \R^n)$ is a reflexive Banach space, for all $\gamma \in (0,1)$, there exists $f_0 \in \dL^{\frac{1}{1-\gamma}}_{\lambda_\star}([0,+\infty); \R^n)$, with $f_0(t) \in \R^n_+$, for all $t \geq 0$, and a subsequence (still denoted by $C_m^{1-\gamma}$), such that
\begin{equation*}
C_m^{1-\gamma}\rightarrow f_0, \quad \text{weakly in}\,\, \dL^{\frac{1}{1-\gamma}}_{\lambda_\star}([0,+\infty); \R^n), \text{ as } m \to \infty.
\end{equation*}
Setting $C_0(t) \coloneqq f_0(t)^{\frac{1}{1-\gamma}}$, $t \geq 0$, we get~\eqref{eq:_bound_c_n_cons1}.

To prove~\eqref{eq:_bound_c_n_cons2}, let us define, for fixed $w \in \R^n$ and $t > 0$, the functional
\begin{equation*}
R_w(f) \coloneqq \int_0^t \inprod{\de^{(t-s)(L+A)} N f(s)^{\frac{1}{1-\gamma}}}{w} \, \dd s, \quad f \in \dL^{\frac{1}{1-\gamma}}_{\lambda_\star}([0,+\infty); \R^n)
\end{equation*}
It is not hard to show that $R_w$ is convex and strongly continuous. Therefore, by \citep[Corollary~3.9]{brezis2011functional} $R_w$ is weakly lower semicontinuous, i.e., for each $w\in \R^n$,
\begin{equation*}
R_w(C_0^{1-\gamma}) \leq \liminf_{m\to\infty} R_w(C_m^{1-\gamma}),
\end{equation*}
which coincides with~\eqref{eq:_bound_c_n_cons2}.

Finally, to prove~\eqref{eq:_bound_c_n_cons3}, we note that, for any $m \geq 1$,
\begin{equation*}
J(C_m)= \int_0^{+\infty} \de^{-\rho t} \left(\sum_{i=1}^n p_i \dfrac{C_{m,i}(t)^{1-\gamma}}{1-\gamma}\right) \dd t
= \dfrac{1}{1-\gamma}\int_0^{+\infty} \de^{-\lambda_\star t} \inprod{C_m(t)^{1-\gamma}}{\de^{(\lambda_\star - \rho)t} p} \, \dd t,
\end{equation*}
where $p \coloneqq (p_1, \dots, p_n)$. Thanks to Assumption~\ref{ass:rho_Frob}-\ref{ass:rho}, the function $t \mapsto \de^{(\lambda_\star - \rho)t} p$ belongs to $\dL^{\frac{1}{\gamma}}_{\lambda_\star}([0,+\infty); \R^n)$, which is the dual space of $\dL^{\frac{1}{1-\gamma}}_{\lambda_\star}([0,+\infty); \R^n)$. Hence, by~\eqref{eq:_bound_c_n_cons1},
\begin{equation*}
\lim_{m \to \infty} J(C_m) = \dfrac{1}{1-\gamma}\int_0^{+\infty} \de^{-\lambda_\star t} \inprod{C_0(t)^{1-\gamma}}{\de^{(\lambda_\star - \rho)t} p} \, \dd t = J(C_0). \qedhere
\end{equation*}
\end{proof}

\section{Proof of Proposition~\ref{prop:max_V_2_node}}\label{app:V2nodes}
Fix $k \in \R^n_{++}$. In the case where $A_1 = A_2 = A$ we have that
\begin{equation*}
V_{b^0}(w;k) = \dfrac{1}{1-\gamma}\left(\dfrac{2\gamma}{\rho-A(1-\gamma)}\right)^\gamma (k_1 + k_2)^{1-\gamma}, \quad \forall w \geq 0.
\end{equation*} 
Therefore, point~\ref{prop:max_V_2_node:Aeq} immediately follows from Proposition~\ref{prop:V_Vb0_equiv}-\ref{prop:twonodes_A}.

To prove point~\ref{prop:max_V_2_node:Aneq}, we study first the monotonicity of the value function of the auxiliary problem $V_{b^0}$ with respect to the graph weight $w$. Starting from~\eqref{eq:aux_expl_sol} and recalling the definition of the function $G$ given in~\eqref{eq:G}, we can rewrite $V_{b^0}$ as
\[
    V_{b^0}(w;k)=\frac{\gamma^\gamma}{1-\gamma}F_1(w)F_2(w)F_3(w;k),
\]
where
\begin{equation*}
F_1(w)=\left(\frac{1}{\rho-\lambda_0(w)(1-\gamma)}\right)^\gamma,
\quad F_2(w)=\left(1+G(w)^{1-\frac1\gamma}\right)^\gamma,
\quad F_3(w;k)=\left(k_{1}+k_{2}G(w)\right)^{1-\gamma}.
\end{equation*}
Differentiating with respect to $w$ and after some computations we get
%\begin{multline*}
%    \frac{d}{dw}V(w,k_0)=\frac{\gamma^\gamma}{1-\gamma}\lambda'_1(w)\left(\frac{d}{dw}\left(F_1(\lambda_0(w))\right)F_2(\lambda_0(w))F_3(\lambda_0(w),k_0)+\right.\\
%\left.+F_1(\lambda_0(w))\frac{d}{dw}\left(F_2(\lambda_0(w))\right)F_3(\lambda_0(w),k_0)+\right.\\
%\left.+F_1(\lambda_0(w))F_2(\lambda_0(w))\frac{d}{dw}\left(F_3(\lambda_0(w),k_0)\right)\right)
%\end{multline*}
%By doing all the calculations, we rewrite the derivative as,
%\begin{multline*}
% \frac{\partial}{\partial w}V_{b^0}(w;k)=\gamma^\gamma F_1(w)F_2(w)F_3(w)\lambda'_0(w)\times\\
%\times\left(\frac{\gamma}{\rho-\lambda_0(w)(1-\gamma)}+\frac{A_2-A_1}{\left(A_2-\lambda_0(w)\right)^2}\left(-\frac{(1+G(\lambda(w))^{1-\frac{1}{\gamma}})^{-1}}{G(\lambda_0(w)^{\frac{1}{\gamma}})}+k_{0,2}(k_{0,1}+k_{0,2}G(\lambda_0(w))^{-1}\right) \right)
%\end{multline*}
\begin{multline*}
 \frac{\partial}{\partial w}V_{b^0}(w;k)=\gamma^\gamma F_1(w)F_2(w)F_3(w)\lambda'_0(w)\times\\
\times\underbrace{\left\{\frac{\gamma}{\rho-\lambda_0(w)(1-\gamma)}+\frac{A_2-A_1}{\left[A_2-\lambda_0(w)\right]^2} \dfrac{k_2 G(w)^{\frac 1\gamma} - k_1}{\left[k_1 + k_2G(w)\right]\left[G(w) + G(w)^{\frac 1 \gamma}\right]}\right\}}_{\eqqcolon X(w)}.
\end{multline*}
From Proposition~\ref{prop:lambda0dec}, condition~\eqref{eq:twonodes_cond}, and since $G(w) > 1$, for all $w > 0$, we deduce that
\begin{align*}
&\gamma^\gamma F_1(w)F_2(w)F_3(w)\lambda'_0(w) < 0, &
&\frac{\gamma}{\rho-\lambda_0(w)(1-\gamma)} > 0, \\
&\frac{A_2-A_1}{\left[A_2-\lambda_0(w)\right]^2} > 0, &
&\left[k_1 + k_2G(w)\right]\left[G(w) + G(w)^{\frac 1 \gamma}\right] > 0,
\end{align*}
for all $w > 0$. Moreover,
\begin{equation}\label{eq:G_ineq}
k_2 G(w)^{\frac 1\gamma} - k_1 > 0 \quad \Longleftrightarrow \quad G(w) > \left(\dfrac{k_1}{k_2}\right)^\gamma.
\end{equation}
%From Proposition~\ref{prop:lambda0dec}, we deduce that $\gamma^\gamma F_1(w)F_2(w)F_3(w)\lambda'_0(w) < 0$, for all $w > 0$.
%Moreover, condition~\eqref{eq:twonodes_cond} ensures that $\frac{\gamma}{\rho-\lambda_0(w)(1-\gamma)} > 0$ and clearly, since $A_1 < A_2$, we have that $\frac{A_2-A_1}{\left[A_2-\lambda_0(w)\right]^2} > 0$. In addition, since $G(w) > 1$, for all $w > 0$, $\left[k_1 + k_2G(w)\right]\left[G(w) + G(w)^{\frac 1 \gamma}\right] > 0$. Finally,
%\begin{equation}\label{eq:G_ineq}
%k_2 G(w)^{\frac 1\gamma} - k_1 > 0 \quad \Longleftrightarrow \quad G(w) > \left(\dfrac{k_1}{k_2}\right)^\gamma.
%\end{equation}
If $k_1 \leq k_2$, then~\eqref{eq:G_ineq} is verified for all $w > 0$, because $G(w) > 1$, for all $w > 0$. In this case, we get that $X(w) > 0$, for all $w > 0$, and hence that $w \mapsto V_{b^0}(w;k)$ is decreasing on $(0,+\infty)$. Therefore, combining this fact with Proposition~\ref{prop:V_Vb0_equiv}-\ref{prop:twonodes_A1}, we deduce point~\ref{prop:max_V_2_node:Aneq}-\ref{prop:max_V_2_node:k1_<_k2}.
If, instead, $k_1 > k_2$, the sign of $X(w)$ may change as $w$ varies and, accordingly, the monotonicity of $w \mapsto V_{b^0}(w;k)$ may change. Note that
\begin{equation*}
\lim_{w \to +\infty} X(w) = \dfrac{2\gamma}{2\rho-(A_1+A_2)(1-\gamma)} + \dfrac{2}{A_2-A_1}\dfrac{k_2-k_1}{k_1+k_2} \eqqcolon \overline X.
\end{equation*}
If $\overline X > 0$, i.e., if $\frac{A_2 - A_1}{g_\infty} > \frac{k_2 - k_1}{\overline k}$, then there exists $\widehat w > 0$ such that $w \mapsto V_{b^0}(w;k)$ is decreasing on $(\widehat w, +\infty)$. Therefore, combining this fact with Proposition~\ref{prop:V_Vb0_equiv}-\ref{prop:twonodes_A1}, we deduce that $w \mapsto V(w;k)$ is monotonically decreasing on $(w_\star, +\infty)$, where $w_\star = \max\{\overline w, \widehat w\}$ and $\overline w$ is the constant appearing in Proposition~\ref{prop:V_Vb0_equiv}. Point~\ref{prop:max_V_2_node:Aneq}-\ref{prop:max_V_2_node:k1_>_k2_inc} is analogously proved. \qed

\bibliographystyle{plainnat}
\bibliography{Bibliography}
\end{document}